\documentclass[11pt,twoside]{article}
\usepackage[T1]{fontenc}
\usepackage[utf8]{inputenc}
\usepackage{amssymb,amsmath,amsthm,amsfonts,mathrsfs,hyperref}
\usepackage{times}
\usepackage{enumerate}
\usepackage{cite,titletoc}
\usepackage{graphicx}
\usepackage{float}
\usepackage{epstopdf}
\usepackage{subcaption}

\allowdisplaybreaks
\newcommand{\R}{\mathbb R}

\newcommand{\supp}{\operatorname{supp}}

\newcommand{\one}{\mathbf 1}

\hypersetup{colorlinks=true,linkcolor=blue,citecolor=red,urlcolor=cyan}

\theoremstyle{plain}
\newtheorem{theorem}{Theorem}[section]
\newtheorem{lemma}[theorem]{Lemma}

\newtheorem{corollary}[theorem]{Corollary}

\theoremstyle{definition}

\theoremstyle{remark}
\newtheorem{remark}[theorem]{Remark}

\numberwithin{equation}{section}

\title{Pointwise endpoint limits for nonlocal operators}
\author{Dinghuai Wang\footnote{Wang Dinghuai (\texttt{Wangdh1990@126.com}) is supported by NSFC (No.~12101010). } 
    \vspace{0.5cm}\\
\small
School of Mathematics and Statistics, Anhui Normal University, Wuhu 241002, China.\\
}
\vspace{0.5cm}
\date{}

\begin{document}
\maketitle
\thispagestyle{empty}

\begin{abstract}
We establish two pointwise endpoint limits for nonlocal operators:
one based on normalized integrals and a distributional limit
based on weighted weak-type norms.
As applications, the limits yield pointwise Bourgain--Brezis--Mironescu,
Maz'ya--Shaposhnikova, Brezis--Seeger--Van
Schaftingen--Yung and Gu--Yung formulas, together with
higher-order and mean-oscillation variants. We then investigate endpoint limits for fractional powers generated
by semigroups and for approximation processes, and derive sharp
strong and weak endpoint estimates for operators in harmonic analysis.
Finally, while Dom\'{i}nguez and Milman obtained weak-type estimates on
product spaces for $p>1$ and left the endpoint $p=1$ open
(see [Adv.~Math.~411 (2022), Paper~No.~108774, p.~22]),
we prove a sharp pointwise limit in the parameter variable and obtain two--sided
iterated weak-type estimates at $p=1$.

\vskip 0.2 true cm


\vskip 0.2 true cm
\noindent
\end{abstract}

\tableofcontents

\section{Introduction}

\subsection{Background and motivation}
Let $0<s<1$ and $1<p<\infty$. The homogeneous fractional Sobolev
seminorm on $\mathbb{R}^{n}$ is defined by
\begin{equation}
\|f\|_{\dot W^{s,p}(\mathbb{R}^n)}
=
\left(
\int_{\mathbb{R}^n}\int_{\mathbb{R}^n}
\frac{|f(x)-f(y)|^p}{|x-y|^{n+sp}}
\,dx\,dy
\right)^{1/p}.
\end{equation}
As the fractional parameter $s$ tends to $0$ or $1$,
the asymptotic behavior of this nonlocal quantity has become
a cornerstone of modern nonlocal analysis.
After suitable renormalization,
the fractional seminorm recovers the $L^p$-norm of the function
(when $s\to 0$) or the $L^p$-norm of its gradient (when $s\to 1$).

The first endpoint phenomenon was established by Bourgain, Brezis and
Mironescu \cite{BBM01,Bre02}. For $f\in\dot W^{1,p}(\mathbb{R}^{n})$,
\begin{equation}
\lim_{s\to1^-}
(1-s)
\|f\|_{\dot W^{s,p}(\mathbb{R}^{n})}^{p}
=
\frac{K_{n,p}}{p}
\|\nabla f\|_{L^p(\mathbb{R}^{n})}^{p},
\end{equation}
where
$$
K_{n,p}
=
\int_{\mathbb S^{n-1}}
|e\cdot\theta|^{p}\,d\sigma(\theta)
$$
is independent of the choice of the unit vector $e$. The opposite endpoint $s\to0^+$ was subsequently investigated by
Maz'ya and Shaposhnikova \cite{MS02}. They proved that
\begin{equation}
\lim_{s\to0^+}
s\|f\|_{\dot W^{s,p}(\mathbb{R}^{n})}^{p}
=
\frac{2}{p}|\mathbb S^{n-1}|
\|f\|_{L^p(\mathbb{R}^{n})}^{p}.
\end{equation}

A surprising formula was introduced by Brezis, Van Schaftingen and Yung
\cite{BSY1}, and later studied systematically by Brezis, Seeger,
Van Schaftingen and Yung \cite{BSVY20}, who developed a weak-type framework
that avoids the singular normalization factor appearing in the BBM formula. More precisely, for $1<p<\infty$ and $\gamma\neq0$,
\begin{equation}
\left[
Q_{\gamma/p}u
\right]_{L^{p,\infty}(\mathbb{R}^{2n},\nu_\gamma)}
\leq
C(n,p,\gamma)
\|\nabla u\|_{L^p(\mathbb{R}^{n})},
\end{equation}
where
$$
Q_{\gamma/p}u(x,y)
=
\frac{u(x)-u(y)}
{|x-y|^{1+\gamma/p}}\qquad \text{and} \qquad d\nu_\gamma(x,y)
=
|x-y|^{\gamma-n}\,dx\,dy.
$$
Furthermore, for $\gamma>0$, they obtained
\begin{equation}
\lim_{\lambda\to+\infty}
\lambda^{p}
\nu_\gamma
\left(
\left\{
(x,y):
\frac{|u(x)-u(y)|}
{|x-y|^{1+\gamma/p}}
>\lambda
\right\}
\right)
=
\frac{K_{n,p}}{|\gamma|}
\|\nabla u\|_{L^p(\mathbb{R}^{n})}^{p}.
\end{equation}
Recently, Gu and Yung \cite{GY21} considered the opposite endpoint $s=0$. For $1\le p<\infty$,
\begin{equation}
c_1^{1/p}
\|u\|_{L^p }
\leq
\left[
\frac{u(x)-u(y)}
{|x-y|^{n/p}}
\right]_{L^{p,\infty}(\mathbb R^n\times\mathbb R^n)}
\leq
2c_2^{1/p}
\|u\|_{L^p } .
\end{equation}
Moreover,
\begin{equation}
\lim_{\lambda\to0^+}
\lambda^p
\mathcal L^{2n}
\left(
\left\{
(x,y):
\frac{|u(x)-u(y)|}{|x-y|^{n/p}}
\geq\lambda
\right\}
\right)
=
2 |\mathbb{S}^{n-1}|
\|u\|_{L^p }^p .
\end{equation}

\medskip

These theories have since been
extended in several directions, including general kernels, higher-order
formulas, metric and ball-Banach settings, interpolation methods, and
oscillation functionals
\cite{BN06,BN18,BSSVY24,DGP+24,DLYYZmetric,DLYYZ23,Davila02,DM23,DSSVY23,HuEtAl25,Nguyen06,Ponce04,
Zhao24,ZYY23}.  Related results concern
Orlicz--Sobolev spaces, general domains, homogeneous fractional Sobolev
spaces, Triebel--Lizorkin methods, sharp assumptions for BBM limits, Carnot
groups, and heat-content energies
\cite{ACPS20,BalMohantaRoy20,BrascoGomezVazquez21,BSY23,DavoliDiFrattaPagliari26,GarofaloTralli23,
GennaioliStefani26,Mohanta24}.
Maximal functions associated with smoothness and local oscillations were
studied systematically in \cite{DeVoreSharpley84}.

\medskip

The purpose of this paper is to develop a unified framework for endpoint limits of nonlocal operators,
encompassing both normalized integral and weak-type formulations.

\subsection{BBM-type and BSVY-type limit formulas}

The main observation of this paper is the following.
All these endpoint phenomena can be reduced to the study of a single
auxiliary function $T:(0,\infty)\to\mathbb{R}$.
Its asymptotic behavior near $0$ and $+\infty$ encodes two complementary limits:
$$
\alpha\int_0^\infty T(t)\,t^{\alpha-1}\,dt
\qquad\text{and}\qquad
\lambda^q\, w_\gamma
\Bigl(
\Bigl\{ t>0 : \frac{T(t)}{t^{\gamma/q}}>\lambda \Bigr\}
\Bigr),
\quad \text{where } dw_\gamma(t)=t^{\gamma-1}\,dt.
$$

Based on this observation, we establish the following limit formulas.
These results form a unified framework for pointwise BBM, MS, BSVY and GY type formulas, together with higher-order and mean-oscillation variants.

\begin{theorem}[BBM-type limit formula]
\label{thm:BBM}
Let \(T:(0,\infty)\to\mathbb{R}\) be measurable.

\begin{enumerate}

\item[(I)] \textbf{Behavior at the origin.}
Assume that
\(
\lim_{t\to0^+}T(t)
\)
exists and is finite. If
\begin{equation*}
|T(t)|\leq C t^{-\eta},
\qquad t\geq1,
\end{equation*}
for some constants \(\eta>0\) and \(C>0\), then
\begin{equation}\label{BBM:origin}
\lim_{t\to0^+}T(t)
=
\lim_{\gamma\to0^+}
\gamma
\int_0^\infty
T(t)t^{\gamma-1}\,dt .
\end{equation}

\item[(II)] \textbf{Behavior at infinity.}
Assume that
\(
\lim_{t\to+\infty}T(t)
\)
exists and is finite. If
\begin{equation*}
|T(t)|\leq C t^{\eta},
\qquad 0<t\leq1,
\end{equation*}
for some constants \(\eta>0\) and \(C>0\), then
\begin{equation}
\label{BBM:infinity}
\lim_{t\to+\infty}T(t)
=
\lim_{\gamma\to0^-}
(-\gamma)
\int_0^\infty
T(t)t^{\gamma-1}\,dt .
\end{equation}

\end{enumerate}
\end{theorem}

For \(\gamma\in\mathbb R\setminus\{0\}\), write
\[
w_\gamma(E):=\int_E t^{\gamma-1}\,dt,
\qquad E\subset(0,\infty)\ \text{measurable}.
\]

\begin{theorem}[BSVY-type limit formula]
\label{thm:BSVY}
Let \(\gamma\in\mathbb{R}\setminus\{0\}\), \(q\in(0,\infty)\), and let
\(T(t)\) be a non-negative measurable function such that
\[
\sup_{t>0}T(t)<\infty .
\]

\begin{enumerate}

\item[(I)] \textbf{Behavior at the origin.}
If
\(
\lim_{t\to0^+}T(t)
\)
exists and is finite, then
\begin{align}
\lim_{t\to0^+}T(t)
&=
\gamma^{1/q}
\lim_{\lambda\to+\infty}
\lambda
\,w_\gamma
\left(
\left\{
t>0:
\frac{T(t)}{t^{\gamma/q}}>\lambda
\right\}
\right)^{1/q},
&&\gamma>0,
\label{BSY:origin:pos}
\\
\lim_{t\to0^+}T(t)
&=
(-\gamma)^{1/q}
\lim_{\lambda\to0^+}
\lambda
\,w_\gamma
\left(
\left\{
t>0:
\frac{T(t)}{t^{\gamma/q}}>\lambda
\right\}
\right)^{1/q},
&&\gamma<0 .
\label{BSY:origin:neg}
\end{align}

\item[(II)] \textbf{Behavior at infinity.}
If
\(
\lim_{t\to+\infty}T(t)
\)
exists and is finite, then
\begin{align}
\lim_{t\to+\infty}T(t)
&=
\gamma^{1/q}
\lim_{\lambda\to0^+}
\lambda
\,w_\gamma
\left(
\left\{
t>0:
\frac{T(t)}{t^{\gamma/q}}>\lambda
\right\}
\right)^{1/q},
&&\gamma>0,
\label{BSY:infinity:pos}
\\
\lim_{t\to+\infty}T(t)
&=
(-\gamma)^{1/q}
\lim_{\lambda\to+\infty}
\lambda
\,w_\gamma
\left(
\left\{
t>0:
\frac{T(t)}{t^{\gamma/q}}>\lambda
\right\}
\right)^{1/q},
&&\gamma<0 .
\label{BSY:infinity:neg}
\end{align}

\end{enumerate}

Furthermore, the following universal estimate holds:
\begin{equation}
\label{BSY:universal}
|\gamma|^{1/q}
\sup_{\lambda>0}
\lambda
\,w_\gamma
\left(
\left\{
t>0:
\frac{T(t)}{t^{\gamma/q}}>\lambda
\right\}
\right)^{1/q}
\leq
\sup_{t>0}T(t).
\end{equation}
\end{theorem}

\subsection{Applications}
We next describe the applications that motivate the preceding theory.

\medskip
\noindent\textbf{1. Pointwise BBM, MS, BSVY and GY formulas.}
A first application concerns the pointwise BBM, MS, BSVY and GY
formulas. For example, consider
\[
T_{q,f}(t)
=
\frac{1}{t^q}
\int_{\mathbb{S}^{n-1}}
|f(x+t\theta)-f(x)|^q\,d\theta .
\]
Applying Theorem~\ref{thm:BBM} to this quantity yields, for
\(0<q<\infty\) and \(f\in C_c^1(\mathbb{R}^n)\),
\[
\lim_{s\to1^-}
q^{1/q}(1-s)^{1/q}
\left(
\int_{\mathbb{R}^{n}}
\frac{|f(x)-f(y)|^q}{|x-y|^{n+sq}}
\,dy
\right)^{1/q}
=
K_{n,q}^{1/q}|\nabla f(x)| ,
\]
and
\[
\lim_{s\to0^+}
q^{1/q}s^{1/q}
\left(
\int_{\mathbb{R}^{n}}
\frac{|f(x)-f(y)|^q}{|x-y|^{n+sq}}
\,dy
\right)^{1/q}
=
|\mathbb{S}^{n-1}|^{1/q}|f(x)| .
\]

Similarly, Theorem~\ref{thm:BSVY} provides pointwise weak-type counterparts
of the BSVY and GY formulas. The same argument also applies to higher-order
difference quotients and Frank-type oscillation formulas.

\medskip
\noindent\textbf{2. Semigroups and approximation processes.}
Let \(L\) be a non-negative densely
defined self-adjoint operator on \(L^2(\mathbb{R}^n)\), and let
\(\{e^{-tL}\}_{t>0}\) be the associated heat semigroup. For
\(0<\sigma<1\), the fractional power \(L^\sigma\) admits the representation
\[
L^\sigma f(x)
=
\frac{1}{\Gamma(-\sigma)}
\int_0^\infty
\left(
e^{-tL}f(x)-f(x)
\right)
\frac{dt}{t^{1+\sigma}} .
\]

The semigroup representation places fractional operators naturally within the
framework of Theorem~\ref{thm:BBM}.

\begin{theorem}\label{thm:unified-frac}
Assume that $f\in\operatorname{Dom}(L)$ and
$$
\lim_{t\to0^+}\frac{f(x)-e^{-tL}f(x)}{t}=Lf(x),
\qquad
\lim_{t\to\infty}e^{-tL}f(x)=0,
$$
and that the two functions
\[
t\longmapsto \frac{f(x)-e^{-tL}f(x)}{t},
\qquad t\longmapsto f(x)-e^{-tL}f(x)
\]
satisfy the polynomial bounds required in Theorem~\ref{thm:BBM}(I) and
(II), respectively. Then
\begin{equation}
\label{eq:unified-limit-1}
\lim_{\sigma\to1^-}L^\sigma f(x)=Lf(x), \qquad
\lim_{\sigma\to0^+}L^\sigma f(x)=f(x).
\end{equation}
\end{theorem}

This theorem immediately applies to several fundamental examples, including
the fractional Laplacian
\(
(-\Delta)^\sigma,
\)
the fractional harmonic oscillator
\(
(-\Delta+|x|^2)^\sigma,
\)
and the fractional heat operator
\(
(\partial_t-\Delta)^\sigma .
\)

\begin{corollary}\label{cor:frac}
The conclusions of Theorem~\ref{thm:unified-frac} hold in each of the
following cases:
\begin{enumerate}
\item $L=-\Delta$ and $f\in\mathcal S(\mathbb R^n)$;
\item $L=-\Delta+|x|^2$ and $f\in\mathcal S(\mathbb R^n)$;
\item $L=\partial_t-\Delta_x$ and
      $f\in\mathcal S(\mathbb R^{n+1})$.
\end{enumerate}
\end{corollary}

The BSVY-type theorem also applies to approximation processes. Let
\(\{K_t\}_{t>0}\) be a family of integrable kernels, and let
\(f\in L^1_{\mathrm{loc}}(\mathbb R^n)\). At a fixed Lebesgue point \(x\) of
\(f\), assume that
\[
\lim_{t\to 0^+}K_t*f(x)= f(x)\qquad \text{and} \qquad
\sup_{t>0}|K_t*f(x)|<\infty.
\]

\begin{theorem}\label{thm:AI}
Let \(x\) be a Lebesgue point of \(f\) satisfying the preceding pointwise
hypotheses. Then, for every
\(\gamma\in\mathbb R\setminus\{0\}\) and \(q\in(0,\infty)\),
\begin{align*}
|f(x)|
&=
\gamma^{1/q}
\lim_{\lambda\to\infty}
\lambda\,
w_\gamma
\left(
\left\{
t>0:
|K_t*f(x)|>\lambda t^{\gamma/q}
\right\}
\right)^{1/q},
&&\gamma>0,
\\
|f(x)|
&=
(-\gamma)^{1/q}
\lim_{\lambda\to0^+}
\lambda\,
w_\gamma
\left(
\left\{
t>0:
|K_t*f(x)|>\lambda t^{\gamma/q}
\right\}
\right)^{1/q},
&&\gamma<0.
\end{align*}
In particular, if
\(
\sup_{t>0}|K_t*f(x)|
\leq
C_KMf(x),
\)
then
\[
|f(x)|
\leq
|\gamma|^{1/q}
\sup_{\lambda>0}
\lambda\,
w_\gamma
\left(
\left\{
t>0:
|K_t*f(x)|>\lambda t^{\gamma/q}
\right\}
\right)^{1/q}
\leq
C_KMf(x),
\]
where $M f(x):=\sup_{B(x,r)}\frac{1}{|B(x,r)|}\int_{B(x,r)}|f(y)|dy$.
\end{theorem}

This result applies to a broad class of approximation schemes, including
Poisson kernels,  Gaussian kernels and extension kernels arising from fractional elliptic equations.

\begin{corollary}\label{cor:AI}
The
conclusions of Theorem~\ref{thm:AI} hold for the following approximation
families:
\begin{enumerate}
\item the Poisson kernels
$$
P_t(x)=c_n\frac{t}{(t^2+|x|^2)^{(n+1)/2}};
$$
\item the Caffarelli--Silvestre extension kernels
$$
P_t^{(s)}(x)=c_{n,s}\frac{t^{2s}}{(t^2+|x|^2)^{(n+2s)/2}},
\qquad0<s<1;
$$
\item the Gaussian kernels
$$
G_t(x)=(4\pi t)^{-n/2}e^{-|x|^2/(4t)};
$$
\end{enumerate}
\end{corollary}

\medskip
\noindent\textbf{3. Limiting behavior for Harmonic analysis operators.} The BBM-type and BSVY-type limit formulas also apply to asymptotic limits of classical
operators. For \(\gamma\ne0\), define
\[
W_\gamma(E):=\int_E |x|^{\gamma-n}\,dx.
\]

\begin{theorem}
\label{thm:weak:T}
Assume that there exist $R_0\ge1$ and
$G_f\in L^1(\mathbb S^{n-1})$ such that $\mathbf T f\in L^{p_0}(B(0,R_0))$ for some $p_0>1$,
\begin{equation}
\label{eq:asymptotic-1}
r^n|\mathbf T f(r\theta)|\le G_f(\theta),
\qquad r\ge R_0,\quad\text{for a.e. }\theta\in\mathbb S^{n-1},
\end{equation}
and that
\begin{equation}
\label{eq:asymptotic-3}
A_f(\theta):=\lim_{r\to\infty}r^n|\mathbf T f(r\theta)|
\end{equation}
exists for almost every $\theta$.  Then, for every $\gamma>0$,
$$
\lim_{\lambda\to0^+}\lambda W_\gamma
\left(\left\{x:
\frac{|\mathbf T f(x)|}{|x|^{\gamma-n}}>\lambda\right\}\right)
=\frac1\gamma\int_{\mathbb S^{n-1}}A_f(\theta)\,d\theta.
$$
If $\gamma<0$ , then
$$
\lim_{\lambda\to\infty}\lambda W_\gamma
\left(\left\{x:
\frac{|\mathbf T f(x)|}{|x|^{\gamma-n}}>\lambda\right\}\right)
=\frac1{-\gamma}\int_{\mathbb S^{n-1}}A_f(\theta)\,d\theta.
$$
\end{theorem}

\begin{theorem}
\label{thm:strong:T}
Under \eqref{eq:asymptotic-1}--\eqref{eq:asymptotic-3} and
 assume that there is \(p_0>1\) such that
\(\mathbf T f\in L^{p_0}(B(0,1))\), \(G_f\in L^{p_0}(\mathbb S^{n-1})\), and
\[
\lim_{R\to\infty}\sup_{r\ge R}
\bigl\|r^n|\mathbf T f(r\,\cdot)|-A_f\bigr\|_{L^{p_0}(\mathbb S^{n-1})}=0.
\]
Then
\[
\lim_{p\to1^+}(p-1)\|\mathbf T f\|_{L^p }^p
= \frac1n\int_{\mathbb S^{n-1}}A_f(\theta)\,d\theta.
\]
\end{theorem}

For \(\gamma = n\), the weighted measure reduces to the Lebesgue measure. The theorem therefore includes the limiting weak-type identities for homogeneous singular integrals and maximal operators studied by Janakiraman~\cite{Janakiraman2005}. That is,
\begin{equation*}
\lim_{\lambda\to0^+}\lambda
\bigl|\{x:Mf(x)>\lambda\}\bigr|=\|f\|_{L^1(\mathbb{R}^n)},
\end{equation*}
and
\begin{equation*}
\lim_{\lambda\to0^+}\lambda
\bigl|\{x:|Tf(x)|>\lambda\}\bigr|
=\frac1n\int_{\mathbb S^{n-1}}\Omega(\theta)\,d\theta\,\|f\|_{L^1(\mathbb{R}^n)}
\end{equation*}
where \(T\) is the singular integral operator with kernel \(\Omega\).

\medskip
\noindent\textbf{4. The Dom\'inguez--Milman question.}
Dom\'inguez and Milman~\cite{DM22} developed a general weak-type framework
for one-parameter families of operators.  Let \((X,m)\) be a
\(\sigma\)-finite measure space, and let
$
\{T_t:t>0\}
$
be a family of operators on \(L^p(X,m)\). Assuming that the associated maximal
operator
$
T^*f=\sup_{t>0}|T_tf|
$
is bounded on \(L^p(X,m)\), \(1<p<\infty\), they proved the estimate
$$
\sup_{\lambda>0}
\lambda^p
(m\times w_\gamma)
\left(
\left\{
(x,t):
\frac{|T_tf(x)|}{t^{\gamma/p}}
>\lambda
\right\}
\right)
\leq
\frac{C_p}{\gamma}
\|f\|_{L^p(X,m)}^p .
$$
Their result naturally raises the endpoint $p=1$ question. The product space results answering the question raised by Dom\'inguez and
Milman were obtained in \cite{DaiLiYangYuanZhao26}. In this paper, we consider the iterated weak estimate for $T_{t}f$. 

\begin{theorem}
\label{thm:DM-1}
Let \((X,m)\) be a \(\sigma\)-finite measure space and let \(\gamma>0\). Suppose that
\(T_tf(x)\to f(x)\) as \(t\to0^+\). Then,
\[
\lim_{\lambda\to\infty}\lambda\,
 w_\gamma\left(\left\{t>0:
 \frac{|T_tf(x)|}{t^\gamma}>\lambda\right\}\right)
=\frac{|f(x)|}{\gamma}
\]
and
\[
\left\|\frac{T_tf(x)}{t^\gamma}\right\|_{L_t^{1,\infty}(w_\gamma)}
\le \frac1\gamma T^*f(x).
\]
Consequently,
\[
\left\|\left\|\frac{T_tf(x)}{t^\gamma}\right\|_{L_t^{1,\infty}(w_\gamma)}
\right\|_{L_x^{1,\infty}(X,m)}
\le \frac1\gamma\|T^*f\|_{L^{1,\infty}(X,m)}.
\]
In particular, if \(T^*:L^1(X,m)\to L^{1,\infty}(X,m)\) has norm at most
\(C\), then the right-hand side is bounded by
\(C\|f\|_{L^1(X,m)}/\gamma\).
\end{theorem}

\medskip

The structure of the paper is as follows.
Section~2 proves the pointwise BBM-type and BSVY-type endpoint limits.
Section~3 applies these limits to pointwise first- and higher-order
formulas of BBM, MS,
BSVY, GY, and Frank type.
Section~4 studies the endpoint behavior of fractional powers generated by
semigroups, including the fractional Laplacian, the fractional harmonic
oscillator, and the fractional heat operator. Section~5 treats approximation
processes and verifies the general results for Poisson, Caffarelli--Silvestre,
and Gaussian kernels. Section~6 establishes strong and weak limiting formulas
for maximal operators, homogeneous Calder\'on--Zygmund operators, and
directional maximal operators. Finally, Section~7 proves the endpoint
iterated weak-type estimate for one-parameter families and explains its
connection with the Dom\'inguez--Milman framework.

\section{Proofs of Theorems~\ref{thm:BBM} and \ref{thm:BSVY}}

We first verify the elementary identity: for any $\lambda>0$, $\gamma\neq0$ and $L>0$,
\begin{equation}\label{BSY:eq1}
\lambda^q\, w_\gamma\Bigl(\Bigl\{t>0: \frac{L}{t^{\gamma/q}}>\lambda\Bigr\}\Bigr)=\frac{L^q}{|\gamma|}.
\end{equation}
Without loss of generality, we may assume $\lambda=1$; otherwise replace $L$ by $L/\lambda$.

\medskip

\noindent\textbf{Case $\gamma>0$.} The inequality $\frac{L}{t^{\gamma/q}}>1$ is equivalent to $t<L^{q/\gamma}$. Consequently,
\[
w_\gamma\bigl(\{t>0: \tfrac{L}{t^{\gamma/q}}>1\}\bigr)
=\int_{0}^{L^{q/\gamma}} t^{\gamma-1}\,dt
=\frac{L^q}{\gamma},
\]
which establishes \eqref{BSY:eq1} for $\gamma>0$.

\noindent\textbf{Case $\gamma<0$.} Here $\frac{L}{t^{\gamma/q}}>1$ yields $t>L^{q/\gamma}$, and we obtain
\[
w_\gamma\bigl(\{t>0: \tfrac{L}{t^{\gamma/q}}>1\}\bigr)
=\int_{L^{q/\gamma}}^{\infty} t^{\gamma-1}\,dt
=-\frac{L^q}{\gamma}.
\]
Thus \eqref{BSY:eq1} holds for $\gamma<0$ as well.

\medskip

\begin{proof}[Proof of Theorem~\ref{thm:BBM}]
We first prove part (I). Set $T_0:=\lim_{t\to 0^+}T(t)$. We first treat the case $T_0>0$. Fix $\varepsilon\in(0,T_0/2)$. By definition of $T_0$, there exists $\delta>0$ such that
\begin{equation}\label{eq:T0:close}
T_0-\varepsilon \le T(t) \le T_0+\varepsilon \qquad\text{for all } t\in(0,\delta).
\end{equation}
Define
\[
\lambda := \frac{(T_0+\varepsilon)(T_0-\varepsilon)}{T_0-2\varepsilon}.
\]
Note that $\lambda > T_0+\varepsilon$ and $\lambda$ depends only on $T_0$ and $\varepsilon$. For $\gamma>0$ set
\[
t_0 = t_0(\gamma) := \Bigl( \frac{T_0+\varepsilon}{\lambda} \Bigr)^{1/\gamma}.
\]
Then $0<t_0<1$ and $t_0\to 0^{+}$ as $\gamma\to 0^{+}$. For all sufficiently small $\gamma>0$, we have $t_0 \le \delta$.
Split the integral at $t_0$, then
\[
\gamma\int_{0}^{\infty} T(t) \, t^{\gamma-1}\,dt
= \gamma\int_{t_0}^{\infty} T(t) \, t^{\gamma-1}\,dt
  + \gamma\int_{0}^{t_0} T(t) \, t^{\gamma-1}\,dt
=: I_1 + I_2 .
\]

For $I_1$, using the condition $|T(t)|\le C t^{-\eta}$ and taking $\gamma<\eta$, we obtain
\[
|I_1| \le C  \gamma \int_{t_0}^{\infty} t^{\gamma-1-\eta}\,dt
      = C \frac{\gamma}{\eta-\gamma} \, t_0^{\eta-\gamma}
      \le C  \frac{\gamma}{\eta-\gamma},
\]
because $t_0<1$ and $\eta-\gamma>0$. Consequently,
\begin{equation}\label{eq:I1}
\lim_{\gamma\to 0^{+}} I_1 = 0 .
\end{equation}

For $I_2$, observe that
\[
I_2 = \gamma\int_{0}^{\infty}
\chi_{\bigl\{\frac{T_0+\varepsilon}{t^{\gamma}} > \lambda\bigr\}}(t)\,
T(t) \, t^{\gamma-1}\,dt,
\]
since $\frac{T_0+\varepsilon}{t^{\gamma}} > \lambda$ is equivalent to $t < t_0$. By \eqref{eq:T0:close}, for every $t<t_0$ we have $T(t) \le T_0+\varepsilon < \lambda$. Hence
\[
I_2 \le \gamma\int_{0}^{\infty}
\chi_{\bigl\{\frac{T_0+\varepsilon}{t^{\gamma}} > \lambda\bigr\}}(t)\,
\lambda \, t^{\gamma-1}\,dt
= \lambda \gamma\, w_{\gamma}\Bigl( \Bigl\{ t>0 :
\frac{T_0+\varepsilon}{t^{\gamma}} > \lambda \Bigr\} \Bigr).
\]
Applying \eqref{BSY:eq1} with $q=1$ gives
\[
\lambda \, w_{\gamma}\Bigl( \Bigl\{ t>0 :
\frac{T_0+\varepsilon}{t^{\gamma}} > \lambda \Bigr\} \Bigr)
= \frac{T_0+\varepsilon}{\gamma},
\]
so that
\begin{equation}\label{eq:I2}
I_2 \le T_0+\varepsilon .
\end{equation}
On the other hand, for $t\in(0,t_0)$, \eqref{eq:T0:close} implies $T(t) \ge T_0-\varepsilon > 0$. Therefore,
\begin{equation}\label{eq:lower}
I_2 \ge \gamma\int_{0}^{t_0} (T_0-\varepsilon) \, t^{\gamma-1}\,dt
= (T_0-\varepsilon) \, t_0^{\gamma}
= T_0-2\varepsilon .
\end{equation}

Combining \eqref{eq:I1}, \eqref{eq:I2} and \eqref{eq:lower}, we obtain for every $\varepsilon\in(0,T_0/2)$,
\[
T_0-2\varepsilon \le \liminf_{\gamma\to 0^{+}} \gamma\int_{0}^{\infty} T(t) \, t^{\gamma-1}\,dt
\le \limsup_{\gamma\to 0^{+}} \gamma\int_{0}^{\infty} T(t) \, t^{\gamma-1}\,dt
\le T_0+\varepsilon .
\]
Letting $\varepsilon\to 0^{+}$ completes the proof for $T_0>0$.

\medskip

\noindent\textbf{The case $T_0<0$.}
Replacing $T(t)$ by $-T(t)$ reduces this case to the previous one, since $\lim_{t\to0^+}(-T(t)) = -T_0 > 0$.

\medskip

\noindent\textbf{The case $T_0=0$.}
Repeating the above estimates with the lower bound replaced by $0$ yields
\[
0 \le \gamma \int_{0}^{\infty} T(t) \, t^{\gamma-1}\,dt \le \varepsilon,
\]
and the result follows by first taking $\gamma\to 0^{+}$ and then letting $\varepsilon\to 0^{+}$.

\medskip

For part (II), set $t = 1/s$. Then $dt = -s^{-2} ds$, and for $\gamma<0$ we have
\[
\int_{0}^{\infty} T(t) \, t^{\gamma-1}\,dt
= \int_{\infty}^{0} T(1/s) \, s^{-\gamma-1} \cdot (-s^{-2}) \, ds
= \int_{0}^{\infty} T(1/s) \, s^{-\gamma-1} \, ds.
\]
The hypothesis $\lim_{t\to +\infty} T(t) = T_{L^{\infty}}$ is equivalent to $\lim_{s\to 0^+} T(1/s) = T_{L^{\infty}}$.
Moreover,
$$|T(t)| \le C\, t^{\eta} \Longrightarrow |T(1/s)| \le C\, s^{-\eta}.$$
Applying \eqref{BBM:origin} to the function $\tilde T(s)=T(1/s)$ gives
\[
T_{L^{\infty}} = \lim_{\gamma\to 0^+} \gamma \int_{0}^{\infty} T(1/s) \, s^{\gamma-1}\,ds.
\]
Replacing $\gamma$ by $-\gamma$ (note that $-\gamma\to 0^-$ as $\gamma\to 0^+$), we obtain
\[
T_{L^{\infty}} = \lim_{\gamma\to 0^-} (-\gamma) \int_{0}^{\infty} T(t) \, t^{\gamma-1}\,dt,
\]
which is exactly \eqref{BBM:infinity}.
\end{proof}

\begin{proof}[Proof of Theorem~\ref{thm:BSVY}]
Assume first that \(\gamma>0\) and \(T(t)\to T_0\) as \(t\to0^+\). With
\(t=\lambda^{-q/\gamma}r\),
\begin{align*}
 &\lambda^q w_\gamma\left(\left\{t:
 \frac{T(t)}{t^{\gamma/q}}>\lambda\right\}\right)=\int_0^\infty
 \one_{\{T(\lambda^{-q/\gamma}r)>r^{\gamma/q}\}}
 r^{\gamma-1}\,dr.
\end{align*}
Write \(M=\sup_{t>0}T(t)\), then
$$\int_0^\infty
 \one_{\{T(\lambda^{-q/\gamma}r)>r^{\gamma/q}\}}
 r^{\gamma-1}\,dr\leq \int_0^\infty
 \one_{\{M>r^{\gamma/q}\}}
 r^{\gamma-1}\,dr=\frac{M^q}{\gamma}<\infty.$$
Applying the dominated convergence theorem, we obtain
\begin{align*}
 \lim_{\lambda\to +\infty}\lambda^q w_\gamma\left(\left\{t:
 \frac{T(t)}{t^{\gamma/q}}>\lambda\right\}\right)
 &=\int_0^\infty \lim_{\lambda\to +\infty}
 \one_{\{T(\lambda^{-q/\gamma}r)>r^{\gamma/q}\}}
 r^{\gamma-1}\,dr\\
 &=\int_0^\infty
 \one_{\{T_{0}>r^{\gamma/q}\}}
 r^{\gamma-1}\,dr\\
 & =\int_0^{T_0^{q/\gamma}}r^{\gamma-1}\,dr
 =\frac{T_0^q}{\gamma}.
\end{align*}
Taking the \(q\)-th root proves \eqref{BSY:origin:pos}.

If \(\gamma<0\), \(t\to0^+\) as $\lambda\to 0^+$. Since
\(\lambda^{-q/\gamma}\to0\), dominated convergence gives
\[
 \lim_{\lambda\to0^+}\lambda^q
 w_\gamma\left(\left\{t>0:
 \frac{T(t)}{t^{\gamma/q}}>\lambda\right\}\right)
 =\frac{T_0^q}{-\gamma},
\]
which is \eqref{BSY:origin:neg}.

The cases at infinity are similar: for
\(\gamma>0\), let \(\lambda\to0^+\), so that
\(\lambda^{-q/\gamma}\to\infty\); for \(\gamma<0\), let
\(\lambda\to\infty\). This yields \eqref{BSY:infinity:pos} and
\eqref{BSY:infinity:neg}.

Finally, the inclusion
\[
 \left\{t:\frac{T(t)}{t^{\gamma/q}}>\lambda\right\}
 \subset
 \left\{t:\frac{M}{t^{\gamma/q}}>\lambda\right\}
\]
and the identity \eqref{BSY:eq1} imply
\[
 \sup_{\lambda>0}\lambda^q
 w_\gamma\left(\left\{t>0:
 \frac{T(t)}{t^{\gamma/q}}>\lambda\right\}\right)
 \le\frac{M^q}{|\gamma|}.
\]
This is equivalent to \eqref{BSY:universal}.
\end{proof}

\section{Pointwise endpoint formulas}

Let \(x\in\mathbb R^n\). For every non-negative measurable function
\(F\),
\begin{equation*}
 \int_{\mathbb R^n}F(x-y)\,dy
 =\int_{\mathbb S^{n-1}}\int_0^\infty F(x+t\theta)t^{n-1}\,dt\,d\theta.
\end{equation*}
Consequently,
\begin{equation}\label{eq:polar-weighted}
 \int_E|x-y|^{\gamma-n}\,dy
 =\int_{\mathbb S^{n-1}}\int_0^\infty
 \one_E(x+t\theta)t^{\gamma-1}\,dt\,d\theta.
\end{equation}
Formula \eqref{eq:polar-weighted} is the bridge from the one-dimensional
measure \(w_\gamma\) to the weighted level sets used in the BSVY and
Gu--Yung formulas.

For \(q>0\), set
\begin{equation*}
 K_{n,q}:=\int_{\mathbb S^{n-1}}|e\cdot\theta|^q\,d\theta,
 \qquad e\in\mathbb S^{n-1}.
\end{equation*}
Rotational invariance shows that the value is independent of \(e\), and
\[
 \int_{\mathbb S^{n-1}}|v\cdot\theta|^q\,d\theta
 =K_{n,q}|v|^q.
\]
When desired, one may write
\[
 K_{n,q}=2\pi^{(n-1)/2}
 \frac{\Gamma((q+1)/2)}{\Gamma((n+q)/2)}.
\]
In particular, for every
\(v\in\mathbb R^n\), rotational invariance and homogeneity yield
\begin{equation}\label{eq:rotational-K}
 \int_{\mathbb S^{n-1}}|v\cdot\theta|^q\,d\theta
 =K_{n,q}|v|^q.
\end{equation}
We keep the integral notation because it also has a direct higher-order
analogue for symmetric tensors.

\medskip

All pointwise identities in this section follow the same pattern. The two basic types are
\begin{equation}\label{BSVY:eq1}
 A_{x,\theta}(t)=\frac{|f(x+t\theta)-f(x)|}{t},
 \qquad
 B_{x,\theta}(t)=|f(x+t\theta)-f(x)|.
\end{equation}
If \(f\in C_c^1\), then
\begin{equation}\label{BSVY:eq2}
\lim_{t\to0^+} A_{x,\theta}(t)=|\nabla f(x)\cdot\theta|\qquad \text{and} \qquad
\lim_{t\to+\infty} B_{x,\theta}(t)=|f(x)|.
\end{equation}
The fundamental theorem of calculus gives
\begin{equation}\label{BSVY:eq3}
 A_{x,\theta}(t)
 \le \frac1t\int_0^t|\nabla f(x+r\theta)|\,dr
 \le M_\theta(|\nabla f|)(x).
\end{equation}
Here
\[
 M_\theta g(x):=\sup_{t>0}\frac1t\int_0^t|g(x+r\theta)|\,dr
\]
denotes the directional maximal operator.

\subsection{Pointwise BBM formula}

For \(0<s<1\), set
\[
\mathcal D_{s,q}f(x):=
\left(\int_{\mathbb R^n}\frac{|f(x)-f(y)|^q}{|x-y|^{n+sq}}\,dy\right)^{1/q}.
\]
The following is the pointwise counterpart of the BBM formula
\cite{BBM01,Bre02}; related extensions appear in
\cite{ACPS20,BSY23,DGP+24}.

\begin{theorem}
Let \(0<q<\infty\) and \(f\in C_c^1(\mathbb R^n)\). Then, for every \(x\),
\begin{equation*}
\lim_{s\to1^-}\bigl(q(1-s)\bigr)^{1/q}\mathcal D_{s,q}f(x)
=K_{n,q}^{1/q}|\nabla f(x)|.
\end{equation*}
\end{theorem}

\begin{proof}
Define
\[
T(t)=t^{-q}\int_{\mathbb S^{n-1}}|f(x+t\theta)-f(x)|^q\,d\theta.
\]
Dominated convergence, \eqref{BSVY:eq1} and \eqref{BSVY:eq2} give
\begin{align*}
\lim_{t\to0^+}T(t)&=\int_{\mathbb S^{n-1}}\lim_{t\to 0^+}\frac{|f(x+t\theta)-f(x)|^q}{t^q}\,d\theta\\
&=\int_{\mathbb S^{n-1}}\lim_{t\to 0^+}A_{x,\theta}^q\,d\theta=K_{n,q}|\nabla f(x)|^q.
\end{align*}
Moreover \(T(t)\lesssim t^{-q}\) for large \(t\). For \(\gamma=q(1-s)\), one has
\begin{equation*}
q(1-s)\mathcal D_{s,q}f(x)^q
=\gamma\int_0^\infty T(t)t^{\gamma-1}\,dt.
\end{equation*}
Theorem~\ref{thm:BBM}(I) yields
\begin{equation*}
\lim_{\gamma\to0^+}\gamma\int_0^\infty T(t)t^{\gamma-1}\,dt
=\lim_{t\to0^+}T(t).
\end{equation*}
Combining the last three displays proves the claim.
\end{proof}

\begin{remark}
We point out that the pointwise BBM formula established above extends the result obtained by Claros and P\'{e}rez in \cite[Lemma~2.1]{CP25} in two directions: it treats the full range \(0<q<\infty\) instead of just \(q=1\), and it requires only \(f\in C_c^1(\mathbb{R}^n)\) rather than \(f\in C_c^2(\mathbb{R}^n)\).
\end{remark}

\subsection{Pointwise MS formula}
In this subsection we prove a pointwise version of the MS formula. The proof relies on the behaviour of the fractional gradient as \(t \to +\infty\), which falls precisely within the scope of Theorem~\ref{thm:BBM}(II).

\begin{theorem}
Let \(0<q<\infty\) and \(f\in C_c^\eta(\mathbb R^n)\), \(0<\eta\le1\).
Then, for every \(x\),
\begin{equation}
\label{eq:MS-pointwise-result}
\lim_{s\to0^+}(qs)^{1/q}\mathcal D_{s,q}f(x)
=|\mathbb S^{n-1}|^{1/q}|f(x)|.
\end{equation}
\end{theorem}

\begin{proof}
Fix \(x\in\mathbb R^n\) and define
\[
 T(t):=\int_{\mathbb S^{n-1}}
 |f(x+t\theta)-f(x)|^q\,d\theta,
 \qquad t>0.
\]
Since \(f\) has compact support, for every \(\theta\in\mathbb S^{n-1}\),
\[
 f(x+t\theta)\longrightarrow0\qquad(t\to\infty).
\]
Moreover, \(|f(x+t\theta)-f(x)|^q\le(2\|f\|_{L^{\infty}})^q\), and hence
spherical dominated convergence gives
\begin{equation}
\label{eq:MS-profile-limit}
 \lim_{t\to\infty}T(t)
 =\int_{\mathbb S^{n-1}}|f(x)|^q\,d\theta
 =|\mathbb S^{n-1}|\,|f(x)|^q.
\end{equation}
For \(0<t\le1\), the H\"older continuity of \(f\) yields
\begin{equation*}
 T(t)
 \le |\mathbb S^{n-1}|[f]_{C^\eta}^q t^{q\eta}.
\end{equation*}
Thus the profile satisfies the small-scale hypothesis of
Theorem~\ref{thm:BBM}(II).

Passing to polar coordinates, we obtain
\begin{align*}
 \mathcal D_{s,q}f(x)^q
 &=\int_{\mathbb S^{n-1}}\int_0^\infty
   |f(x+t\theta)-f(x)|^q t^{-sq-1}\,dt\,d\theta\\
 &=\int_0^\infty T(t)t^{-sq-1}\,dt.
\end{align*}
Set \(\gamma=-sq<0\). Then
\begin{equation*}
 qs\mathcal D_{s,q}f(x)^q
 =(-\gamma)\int_0^\infty T(t)t^{\gamma-1}\,dt.
\end{equation*}
As \(s\to0^+\), one has \(\gamma\to0^-\). Applying
Theorem~\ref{thm:BBM}(II) and using \eqref{eq:MS-profile-limit},
\[
 \lim_{s\to0^+}qs\mathcal D_{s,q}f(x)^q
 =|\mathbb S^{n-1}|\,|f(x)|^q.
\]
Taking the positive \(q\)-th root proves
\eqref{eq:MS-pointwise-result}.
\end{proof}

\subsection{Pointwise BSVY formula}

For fixed \(x\) define
\[
E_x^{(1)}(f,\lambda,q,\gamma)
:=\{y:|f(x)-f(y)|>\lambda|x-y|^{1+\gamma/q}\}.
\]

\begin{theorem}
Let \(\gamma\ne0\), \(0<q<\infty\), and \(f\in C_c^1(\mathbb R^n)\).
For every \(x\),
\[
\begin{cases}
\displaystyle
\lim_{\lambda\to\infty}\lambda
\left(\int_{E_x^{(1)}(f,\lambda,q,\gamma)}|x-y|^{\gamma-n}\,dy\right)^{1/q}
=\dfrac{K_{n,q}^{1/q}}{\gamma^{1/q}}|\nabla f(x)|,
&\gamma>0,\\[12pt]
\displaystyle
\lim_{\lambda\to0^+}\lambda
\left(\int_{E_x^{(1)}(f,\lambda,q,\gamma)}|x-y|^{\gamma-n}\,dy\right)^{1/q}
=\dfrac{K_{n,q}^{1/q}}{(-\gamma)^{1/q}}|\nabla f(x)|,
&\gamma<0.
\end{cases}
\]
Moreover,
\[
\sup_{\lambda>0}\lambda
\left(\int_{E_x^{(1)}(f,\lambda,q,\gamma)}|x-y|^{\gamma-n}\,dy\right)^{1/q}
\le |\gamma|^{-1/q}
\left(\int_{\mathbb S^{n-1}}M_\theta(|\nabla f|)(x)^q\,d\theta\right)^{1/q},
\]
\end{theorem}

\begin{proof}
For \(\theta\in\mathbb S^{n-1}\), put
\[
 T(t):=\frac{|f(x+t\theta)-f(x)|}{t},
 \qquad t>0.
\]
The polar-coordinate identity \eqref{eq:polar-weighted} gives
\begin{equation}
\label{eq:BSVY-pointwise-polar}
 \int_{E_x^{(1)}(f,\lambda,q,\gamma)}|x-y|^{\gamma-n}\,dy
 =\int_{\mathbb S^{n-1}}
 w_\gamma\left(\left\{t>0:
 \frac{T(t)}{t^{\gamma/q}}>\lambda\right\}\right)d\theta.
\end{equation}
The directional derivative and the fundamental theorem of calculus yield
\begin{equation*}
 \lim_{t\to0^+}T(t)=|\nabla f(x)\cdot\theta|,
 \qquad
 T(t)
 \le M_\theta(|\nabla f|)(x).
\end{equation*}
Assume first that \(\gamma>0\). By
Theorem~\ref{thm:BSVY}(I), for every \(\theta\),
\[
 \lambda^q w_\gamma\left(\left\{t:
 \frac{T(t)}{t^{\gamma/q}}>\lambda\right\}\right)
 \longrightarrow
 \frac{|\nabla f(x)\cdot\theta|^q}{\gamma}
 \qquad(\lambda\to\infty).
\]
The universal estimate and \eqref{BSVY:eq3} imply that
\[
 0\le \lambda^q w_\gamma\left(\left\{t:
 \frac{T(t)}{t^{\gamma/q}}>\lambda\right\}\right)
 \le\frac{M_\theta(|\nabla f|)(x)^q}{\gamma}.
\]
For \(f\in C_c^1\), the right-hand side is integrable on the sphere. Hence,
by dominated convergence and \eqref{eq:rotational-K},
\begin{align*}
 &\lim_{\lambda\to\infty}\lambda^q
   \int_{E_x^{(1)}(f,\lambda,q,\gamma)}|x-y|^{\gamma-n}\,dy\\
 &=\frac1\gamma\int_{\mathbb S^{n-1}}
   |\nabla f(x)\cdot\theta|^q\,d\theta
 =\frac{K_{n,q}}\gamma|\nabla f(x)|^q.
\end{align*}
Taking the \(q\)-th root proves the first formula.

If \(\gamma<0\), the same argument uses the origin limit
\(\lambda\to0^+\) in Theorem~\ref{thm:BSVY}(I), and gives
\[
 \lim_{\lambda\to0^+}\lambda^q
   \int_{E_x^{(1)}(f,\lambda,q,\gamma)}|x-y|^{\gamma-n}\,dy
 =\frac{K_{n,q}}{-\gamma}|\nabla f(x)|^q.
\]
Finally, integrate the universal bound in \(\theta\) in
\eqref{eq:BSVY-pointwise-polar} and take the \(q\)-th root:
\[
 \sup_{\lambda>0}\lambda
 \left(\int_{E_x^{(1)}(f,\lambda,q,\gamma)}|x-y|^{\gamma-n}\,dy\right)^{1/q}
 \le |\gamma|^{-1/q}
 \left(\int_{\mathbb S^{n-1}}
 M_\theta(|\nabla f|)(x)^q\,d\theta\right)^{1/q}.
\]
\end{proof}

\begin{remark}
As shown by Calder\'{o}n \cite[Lemma 7]{Calderon72}, \(M_\theta g(x)\) is bounded on \(L^p (1<p<\infty)\)  uniformly in $\theta$. Thus, by the preceding theorem we can obtain the $L^p$-BSVY formula.
\end{remark}

\subsection{Pointwise GY formula}

Define
\[
E_x^{(0)}(f,\lambda,q,\gamma)
:=\{y:|f(x)-f(y)|>\lambda|x-y|^{\gamma/q}\}.
\]

\begin{theorem}
Let \(\gamma\ne0\), \(0<q<\infty\), and \(f\in C_c^\eta(\mathbb R^n)\).
Then
\[
\begin{cases}
\displaystyle
\lim_{\lambda\to0^+}\lambda
\left(\int_{E_x^{(0)}(f,\lambda,q,\gamma)}|x-y|^{\gamma-n}\,dy\right)^{1/q}
=\dfrac{|\mathbb S^{n-1}|^{1/q}}{\gamma^{1/q}}|f(x)|,
&\gamma>0,\\[12pt]
\displaystyle
\lim_{\lambda\to\infty}\lambda
\left(\int_{E_x^{(0)}(f,\lambda,q,\gamma)}|x-y|^{\gamma-n}\,dy\right)^{1/q}
=\dfrac{|\mathbb S^{n-1}|^{1/q}}{(-\gamma)^{1/q}}|f(x)|,
&\gamma<0.
\end{cases}
\]
\end{theorem}

\begin{proof}
For \(\theta\in\mathbb S^{n-1}\), define
\[
 T(t):=|f(x+t\theta)-f(x)|.
\]
Polar coordinates show that
\begin{equation*}
 \int_{E_x^{(0)}(f,\lambda,q,\gamma)}|x-y|^{\gamma-n}\,dy
 =\int_{\mathbb S^{n-1}}
 w_\gamma\left(\left\{t>0:
 \frac{T(t)}{t^{\gamma/q}}>\lambda\right\}\right)d\theta.
\end{equation*}
Because \(f\) is compactly supported,
\[
 \lim_{t\to\infty}T(t)=|f(x)|,
 \qquad
 0\le T(t)\le2\|f\|_{L^{\infty}}.
\]
For \(\gamma>0\), Theorem~\ref{thm:BSVY}(II) gives
\[
 \lim_{\lambda\to0^+}\lambda^q w_\gamma\left(\left\{t:
 \frac{T(t)}{t^{\gamma/q}}>\lambda\right\}\right)
 =\frac{|f(x)|^q}{\gamma},
\]
whereas for \(\gamma<0\),
\[
 \lim_{\lambda\to +\infty}\lambda^q w_\gamma\left(\left\{t:
 \frac{T(t)}{t^{\gamma/q}}>\lambda\right\}\right)
 =\frac{|f(x)|^q}{-\gamma}.
\]
In both cases the universal estimate supplies the uniform bound
\[
 \lambda^q w_\gamma\left(\left\{t>0:
 \frac{T(t)}{t^{\gamma/q}}>\lambda\right\}\right)
 \le\frac{(2\|f\|_{L^{\infty}})^q}{|\gamma|}.
\]
Dominated convergence in \(\theta\), followed by taking the \(q\)-th root,
proves the two asserted limits.
\end{proof}

Observe that the radial maximal quantity
\[
 \sup_{t>0}|f(x+t\theta)-f(x)|
\]
need not define an \(L^p\)-bounded function of \(x\).  Consequently, the
pointwise GY formula cannot be integrated by a formal maximal-function
argument.  We now use the following theorem to derive the $L^{p}$-GY formula.

\begin{theorem}
Let \(\gamma>0\), \(0<p<\infty\), and \(f,g\in L^p \).  For
\(\lambda>0\), set
\[
 E_\lambda(f,g)
 :=\bigl\{(x,y)\in\mathbb R^n\times\mathbb R^n:
 |f(x)-g(y)|>\lambda |x-y|^{\gamma/p}\bigr\}.
\]
Then
\[
 \sup_{\lambda>0}\lambda^p
 \iint_{E_\lambda(f,g)}|x-y|^{\gamma-n}\,dx\,dy
 \le \frac{2^p|\mathbb S^{n-1}|}{\gamma}
 \bigl(\|f\|_{L^p}^p+\|g\|_{L^p}^p\bigr).
\]
In particular, if \(f=g\), we obtain
\[
 \sup_{\lambda>0}\lambda^p
 \iint_{E_\lambda(f,f)}|x-y|^{\gamma-n}\,dx\,dy
 \le \frac{2^{p+1}|\mathbb S^{n-1}|}{\gamma}
 \|f\|_{L^p}^p.
\]
\end{theorem}

\begin{proof}
Fix \(\lambda>0\).  From
\[
 |f(x)-g(y)|\le |f(x)|+|g(y)|
\]
we obtain the pointwise inclusion
\[
 E_\lambda(f,g)\subset A_\lambda\cup B_\lambda,
\]
where
\[
 A_\lambda
 :=\left\{(x,y):|f(x)|>\frac{\lambda}{2}|x-y|^{\gamma/p}\right\},
 \qquad
 B_\lambda
 :=\left\{(x,y):|g(y)|>\frac{\lambda}{2}|x-y|^{\gamma/p}\right\}.
\]
Indeed, if a point belongs to neither \(A_\lambda\) nor \(B_\lambda\), then
\[
 |f(x)-g(y)|
 \le |f(x)|+|g(y)|
 \le \lambda |x-y|^{\gamma/p},
\]
so it does not belong to \(E_\lambda(f,g)\).  Hence,
\begin{align*}
 \iint_{E_\lambda(f,g)}|x-y|^{\gamma-n}\,dx\,dy
 &\le I_f(\lambda)+I_g(\lambda),
\end{align*}
with
\[
 I_f(\lambda)
 :=\int_{\mathbb R^n}\int_{\mathbb R^n}
 \mathbf 1_{A_\lambda}(x,y)|x-y|^{\gamma-n}\,dy\,dx
\]
and the analogous definition of \(I_g(\lambda)\).

For a fixed $x\in \R^n$, the set \(A_\lambda\) at \(x\) is exactly the ball
\(B(x,\left(\frac{2|f(x)|}{\lambda}\right)^{p/\gamma}\). Therefore,
\begin{align*}
 \int_{\mathbb R^n}\mathbf 1_{A_\lambda}(x,y)
 |x-y|^{\gamma-n}\,dy
 &=|\mathbb S^{n-1}|\int_0^{\left(\frac{2|f(x)|}{\lambda}\right)^{p/\gamma}}r^{\gamma-1}\,dr\\
 &=\frac{|\mathbb S^{n-1}|}{\gamma}
 \left(\frac{2|f(x)|}{\lambda}\right)^{p}.
\end{align*}
It follows that
\[
 I_f(\lambda)
 =\frac{2^p|\mathbb S^{n-1}|}{\gamma\lambda^p}
 \|f\|_{L^p}^p.
\]
In the same method, we get
\[
 I_g(\lambda)
 =\frac{2^p|\mathbb S^{n-1}|}{\gamma\lambda^p}
 \|g\|_{L^p}^p.
\]
Consequently,
\[
 \lambda^p
 \iint_{E_\lambda(f,g)}|x-y|^{\gamma-n}\,dx\,dy
 \le \frac{2^p|\mathbb S^{n-1}|}{\gamma}
 \bigl(\|f\|_{L^p}^p+\|g\|_{L^p}^p\bigr).
\]
The right-hand side is independent of \(\lambda\); taking the supremum over
\(\lambda>0\) completes the proof.  Notice that the proof uses only the
scalar inequality \(|a-b|\le |a|+|b|\), not the triangle inequality in
\(L^p\), and is therefore valid for every \(0<p<\infty\).
\end{proof}

\subsection{Pointwise Frank-type mean oscillation formula}

For \(1\le r<\infty\), write
\[
 m_f^{(q)}(x,t)
 :=\left(\frac1{|B(x,t)|}\int_{B(x,t)}
 |f(y)-f_{B(x,t)}|^q\,dy\right)^{1/q}.
\]

We now establish the pointwise Frank-type mean oscillation formula.

\begin{theorem}\label{thm:Frank-osc}
Let \(1\le q<\infty\) and \(f\in C_c^1(\mathbb R^n)\). Then, for every
\(x\in\mathbb R^n\),
\begin{equation}\label{eq:main-oscillation}
\lim_{t\to0^+}\frac{m_f^{(q)}(x,t)}{t}
=
\tilde{K}_{n,q}|\nabla f(x)|,
\end{equation}
where
\[
\tilde{K}_{n,q}
:=
\left(
\frac{1}{|B(0,1)|}
\int_{B(0,1)}|e\cdot z|^q\,dz
\right)^{1/q}=\Big(\frac{n}{n+q}K_{n,q}\Big)^{1/q}.
\]
\end{theorem}

\begin{proof}
Fix \(x\in\mathbb R^n\), we write
\[
f(y)
=
f(x)+\nabla f(x)\cdot(y-x)+\eta_x(y),
\]
where
\[
|\eta_x(y)|
\le
\omega_x(|y-x|)|y-x|,
\qquad
\omega_x(t)
:=
\sup_{0<|y-x|\le t}
\frac{|f(y)-f(x)-\nabla f(x)\cdot(y-x)|}{|y-x|},
\]
and \(\omega_x(t)\to0\) as \(t\to0^+\).

Since the ball \(B(x,t)\) is symmetric about \(x\), then
\[
\frac{1}{|B(x,t)|}
\int_{B(x,t)}
\nabla f(x)\cdot(y-x)\,dy
=0.
\]
Hence
\[
f_{B(x,t)}=f(x)+\rho_x(t),
\qquad
\rho_x(t)
:=
\frac{1}{|B(x,t)|}
\int_{B(x,t)}\eta_x(y)\,dy.
\]
For \(y\in B(x,t)\), one has
\[
|\eta_x(y)|\le \omega_x(t)t,
\qquad
|\rho_x(t)|\le \omega_x(t)t.
\]
Therefore,
\[
f(y)-f_{B(x,t)}
=
A_t(y)+R_t(y),
\]
where
\[
A_t(y):=\nabla f(x)\cdot(y-x),
\qquad
R_t(y):=\eta_x(y)-\rho_x(t),
\]
and
\[
|R_t(y)|\le 2\omega_x(t)t
\qquad\text{for }y\in B(x,t).
\]

The reverse triangle inequality in \(L^q(B(x,t))\) now gives
\begin{align}
&\left|
\frac{m_f^{(q)}(x,t)}{t}
-
\frac{1}{t}
\left(
\frac{1}{|B(x,t)|}
\int_{B(x,t)}|A_t(y)|^q\,dy
\right)^{1/q}
\right|
\nonumber\\
&\hspace{4em}\le
\frac{1}{t}
\left(
\frac{1}{|B(x,t)|}
\int_{B(x,t)}|R_t(y)|^q\,dy
\right)^{1/q}
\le
2\omega_x(t).
\label{eq:osc-error}
\end{align}

With the change of variables
\(y=x+tz\), using \(dy=t^n\,dz\) and
\(|B(x,t)|=t^n|B(0,1)|\), we obtain
\[
\frac{1}{t}
\left(
\frac{1}{|B(x,t)|}
\int_{B(x,t)}
|\nabla f(x)\cdot(y-x)|^q\,dy
\right)^{1/q}
=
\left(
\frac{1}{|B(0,1)|}
\int_{B(0,1)}
|\nabla f(x)\cdot z|^q\,dz
\right)^{1/q}.
\]
By rotational invariance of the unit ball and Lebesgue measure,
\[
\int_{B(0,1)}
|\nabla f(x)\cdot z|^q\,dz
=
|\nabla f(x)|^q
\int_{B(0,1)}\Big|\frac{\nabla f(x)}{|\nabla f(x)|}\cdot z\Big|^q\,dz.
\]
Consequently,
\begin{equation}\label{eq:linear-oscillation}
\frac{1}{t}
\left(
\frac{1}{|B(x,t)|}
\int_{B(x,t)}|A_t(y)|^q\,dy
\right)^{1/q}
=
\tilde{K}_{n,q}|\nabla f(x)|.
\end{equation}

Combining \eqref{eq:osc-error} and \eqref{eq:linear-oscillation}, we find
\[
\left|
\frac{m_f^{(q)}(x,t)}{t}
-
\tilde{K}_{n,q}|\nabla f(x)|
\right|
\le
2\omega_x(t).
\]
Letting \(t\to0^+\) proves \eqref{eq:main-oscillation}.
\end{proof}

For \(\lambda>0\), define the radial level set
\[
 \widetilde E_{x,r}(f,\lambda,q,\gamma)
 :=\left\{y\in\mathbb R^n:
 m_f^{(r)}(x,|x-y|)>
 \lambda |x-y|^{1+\gamma/q}\right\}.
\]

\begin{theorem}\label{Frank-pointwise}
Let \(1\le r<\infty\), \(\gamma\ne0\), \(0<q<\infty\), and
\(f\in C_c^1(\mathbb R^n)\).  Then, for every \(x\in\mathbb R^n\),
\[
\begin{cases}
\displaystyle
\lim_{\lambda\to\infty}\lambda
\left(\int_{\widetilde E_{x,r}(f,\lambda,q,\gamma)}
|x-y|^{\gamma-n}\,dy\right)^{1/q}
=\dfrac{|\mathbb S^{n-1}|^{1/q}\tilde{K}_{n,r}}{\gamma^{1/q}}
 |\nabla f(x)|,
&\gamma>0,\\[14pt]
\displaystyle
\lim_{\lambda\to0^+}\lambda
\left(\int_{\widetilde E_{x,r}(f,\lambda,q,\gamma)}
|x-y|^{\gamma-n}\,dy\right)^{1/q}
=\dfrac{|\mathbb S^{n-1}|^{1/q}\tilde{K}_{n,r}}{(-\gamma)^{1/q}}
 |\nabla f(x)|,
&\gamma<0.
\end{cases}
\]
Moreover,
\[
\begin{split}
&\sup_{\lambda>0}\lambda
\left(\int_{\widetilde E_{x,r}(f,\lambda,q,\gamma)}
|x-y|^{\gamma-n}\,dy\right)^{1/q}\le C_{n,r}|\mathbb S^{n-1}|^{1/q}|\gamma|^{-1/q}
 \bigl[M (|\nabla f|^r)(x)\bigr]^{1/r}.
\end{split}
\]
\end{theorem}

\begin{proof}
Set
\[
 U_{x,r}(t):=\frac{m_f^{(r)}(x,t)}{t},
 \qquad t>0.
\]
Since the defining inequality for \(\widetilde E_{x,r}\) depends only on
\(|x-y|\), polar coordinates give
\[
 \int_{\widetilde E_{x,r}(f,\lambda,q,\gamma)}
 |x-y|^{\gamma-n}\,dy
 =|\mathbb S^{n-1}|\,
 w_\gamma\left(\left\{t>0:
 \frac{U_{x,r}(t)}{t^{\gamma/q}}>\lambda\right\}\right).
\]
By Theorem~\ref{thm:Frank-osc},
\[
 \lim_{t\to0^+}U_{x,r}(t)=\tilde{K}_{n,r}|\nabla f(x)|.
\]
On the other hand, the \(L^r\)-Poincar\'e inequality on a ball gives
\[
 \left(\frac1{|B(x,t)|}\int_{B(x,t)}
 |f(y)-f_{B(x,t)}|^r\,dy\right)^{1/r}
 \le C_{n,r}\cdot t
 \left(\frac1{|B(x,t)|}\int_{B(x,t)}|\nabla f(y)|^r\,dy\right)^{1/r}.
\]
Since \(x\) is the center of \(B(x,t)\), it follows that
\[
 0\le U_{x,r}(t)
 \le C_{n,r}\bigl[M(|\nabla f|^r)(x)\bigr]^{1/r}
 \qquad(t>0).
\]
Thus \(U_{x,r}\) is bounded and satisfies all hypotheses of
Theorem~\ref{thm:BSVY}(I).  If \(\gamma>0\), that theorem yields
\[
 \lim_{\lambda\to\infty}\lambda^q
 w_\gamma\left(\left\{t>0:
 \frac{U_{x,r}(t)}{t^{\gamma/q}}>\lambda\right\}\right)
 =\frac{\tilde{K}_{n,r}^q|\nabla f(x)|^q}{\gamma}.
\]
If \(\gamma<0\), it yields instead
\[
 \lim_{\lambda\to0^+}\lambda^q
 w_\gamma\left(\left\{t>0:
 \frac{U_{x,r}(t)}{t^{\gamma/q}}>\lambda\right\}\right)
 =\frac{\tilde{K}_{n,r}^q|\nabla f(x)|^q}{-\gamma}.
\]
Multiplying by \(|\mathbb S^{n-1}|\) and taking the positive \(q\)-th root
proves both limits.  Finally, the universal estimate in
Theorem~\ref{thm:BSVY} shows
\begin{align*}
 &\sup_{\lambda>0}\lambda^q
 \int_{\widetilde E_{x,r}(f,\lambda,q,\gamma)}
 |x-y|^{\gamma-n}\,dy\\
 &\qquad\le
 \frac{|\mathbb S^{n-1}|}{|\gamma|}
 \sup_{t>0}U_{x,r}(t)^q
\le
 \frac{C_{n,r}^q|\mathbb S^{n-1}|}{|\gamma|}
 \bigl[M(|\nabla f|^r)(x)\bigr]^{q/r}.
\end{align*}
Taking the \(q\)-th root completes the proof.
\end{proof}

\begin{remark}
Using
\[
|\mathbb S^{n-1}|
=
\frac{2\pi^{n/2}}{\Gamma\!\left(\frac n2\right)}
\]
and
\[
\widetilde K_{n,r}
=
\left(
\frac{
2\pi^{\frac{n-1}{2}}
\Gamma\!\left(\frac{r+1}{2}\right)
}{
\Gamma\!\left(\frac{n+r}{2}\right)
}
\right)^{1/r},
\]
the constant in the preceding formula can be expressed as
\[
\frac{|\mathbb S^{n-1}|\widetilde K_{n,r}^q{}}
{|\gamma|}
=
|\gamma|^{-1}
\left(
\frac{2\pi^{n/2}}
{\Gamma\!\left(\frac n2\right)}
\right)
\left(
\frac{
2\pi^{\frac{n-1}{2}}
\Gamma\!\left(\frac{r+1}{2}\right)
}{
\Gamma\!\left(\frac{n+r}{2}\right)
}
\right)^{q/r}.
\]
In the particular case \(q=p\), \(\gamma=-p\), and \(r=1\), the constant is
\[
\frac{1}{p}
\left(
\frac{1}{\sqrt{\pi}}\,
\frac{
\Gamma\!\left(\frac{n+2}{2}\right)
}{
\Gamma\!\left(\frac{n+3}{2}\right)
}
\right)^p,
\]
in agreement with \cite[Theorem~1]{Frank22}.
\end{remark}

\begin{remark}
A closely related oscillation is obtained by averaging over two points of the
same ball.  For \(1\le r<\infty\), set
\[
 \widetilde m_f^{(r)}(x,t)
 :=\left(\frac1{|B(x,t)|^2}
 \iint_{B(x,t)^2}|f(y)-f(z)|^r\,dy\,dz\right)^{1/r}
\]
and
\[
  k_{n,r}^r
 :=\frac1{|B_1|^2}\iint_{B_1^2}|e\cdot (z-w)|^r\,dz\,dw, \qquad \qquad e\in \mathbb{S}^{n-1}.
\]
Scaling \(y=x+tz\), \(z=x+tw\), followed by the first-order Taylor
expansion, gives
\begin{align*}
 \left(\frac{\widetilde m_f^{(r)}(x,t)}{t}\right)^r
 &=\frac1{|B_1|^2}\iint_{B_1^2}
 \left|\nabla f(x)\cdot(z-w)+e_t(z,w)\right|^r\,dz\,dw,
\end{align*}
where \(\|e_t\|_{L^\infty(B_1^2)}\to0\).  Hence
\[
 \lim_{t\to0^+}\frac{\widetilde m_f^{(r)}(x,t)}{t}
 = k_{n,r}|\nabla f(x)|.
\]
The two oscillations are quantitatively comparable.  Jensen's inequality and
the triangle inequality imply
\[
 m_f^{(r)}(x,t)
 \le \widetilde m_f^{(r)}(x,t)
 \le 2m_f^{(r)}(x,t).
\]

If one defines the level set by replacing \(m_f^{(r)}\) with
\(\widetilde m_f^{(r)}\), the proof of Theorem~\ref{Frank-pointwise} also holds.  The limiting constant becomes \( k_{n,r}\), both signs
of \(\gamma\) are allowed, and the universal bound follows from
\(\widetilde m_f^{(r)}\le2m_f^{(r)}\) and the same Poincar\'e estimate.
\end{remark}

\subsection{Higher-order BBM formula}

We now pass from first differences to finite differences of arbitrary order.
At this level the limiting object is no longer a vector but the symmetric
\(k\)-tensor \(D^kf(x)\).  We therefore recall the tensor notation explicitly
before stating the endpoint formula.

Let \(\operatorname{Sym}^k(\mathbb R^n)\) denote the space of symmetric
\(k\)-linear forms on \((\mathbb R^n)^k\).  If
\(A\in\operatorname{Sym}^k(\mathbb R^n)\), write
\[
 A[v_1,\ldots,v_k]
 :=\sum_{i_1,\ldots,i_k=1}^n
 A_{i_1\cdots i_k}(v_1)_{i_1}\cdots(v_k)_{i_k},
 \qquad
 A[v^k]:=A[v,\ldots,v].
\]
For a \(C^k\)-function,
\[
 D^kf(x)[v_1,\ldots,v_k]
 =\sum_{i_1,\ldots,i_k=1}^n
 \partial_{i_1}\cdots\partial_{i_k}f(x)
 (v_1)_{i_1}\cdots(v_k)_{i_k},
\]
so that
\[
 D_\theta^kf(x)=D^kf(x)[\theta^k].
\]
For \(0<q<\infty\), define the rotationally invariant tensor gauge
\[
 \mathcal A_{n,k,q}(A)
 :=\left(\int_{\mathbb S^{n-1}}|A[\theta^k]|^q\,d\theta\right)^{1/q}.
\]
When \(q\ge1\), this is a norm on \(\operatorname{Sym}^k(\mathbb R^n)\); when
\(0<q<1\), it is a  quasi-norm. For \(k=1\), the definition reduces to
\[
 \mathcal A_{n,1,q}(v)=K_{n,q}^{1/q}|v|.
\]
For \(k\ge2\), no single scalar multiple of the norm represents
\(\mathcal A_{n,k,q}\) in general, because different trace components of a
symmetric tensor contribute differently to the spherical polynomial.

For \(h\in\mathbb R^n\), define
\[
 \Delta_h^kf(x)
 :=\sum_{j=0}^k(-1)^{k-j}\binom{k}{j}f(x+jh)
\]
and
\[
 \mathcal D_{s,q}^{(k)}f(x)
 :=\left(\int_{\mathbb R^n}\frac{|\Delta_h^kf(x)|^q}
 {|h|^{n+(k-1+s)q}}\,dh\right)^{1/q},
 \qquad 0<s<1.
\]

\begin{lemma}\label{H-BBM-lim}
Let \(f\in C^k\) in a neighborhood of the segment
\(\{x+u\theta:0\le u\le kt\}\).  Then
\[
 \frac{\Delta_{t\theta}^kf(x)}{t^k}
 =\int_{[0,1]^k}D_\theta^kf
 \bigl(x+t(s_1+\cdots+s_k)\theta\bigr)
 \,ds_1\cdots ds_k.
\]
Consequently, for every \(x\) and \(\theta\in\mathbb S^{n-1}\),
\[
 \lim_{t\to0^+}\frac{\Delta_{t\theta}^kf(x)}{t^k}
 =D_\theta^kf(x)=D^kf(x)[\theta^k].
\]
If \(D^kf\) is continuous near \(x\), the convergence is uniform in
\(\theta\in\mathbb S^{n-1}\).
\end{lemma}

\begin{proof}
For \(k=1\), the fundamental theorem of calculus gives
\[
 \frac{f(x+t\theta)-f(x)}{t}
 =\int_0^1D_\theta f(x+ts\theta)\,ds.
\]
Assume the stated identity holds at order \(k-1\).  Since
\(\Delta_{t\theta}^k=\Delta_{t\theta}(\Delta_{t\theta}^{k-1})\), applying the
order-one identity to the function \(u\mapsto
\Delta_{t\theta}^{k-1}f(u)\) gives
\begin{align*}
 \frac{\Delta_{t\theta}^kf(x)}{t^k}
 &=\frac1{t^{k-1}}\int_0^1
 D_\theta\bigl(\Delta_{t\theta}^{k-1}f\bigr)
 (x+ts_k\theta)\,ds_k\\
 &=\int_{[0,1]^k}D_\theta^kf
 \bigl(x+t(s_1+\cdots+s_k)\theta\bigr)
 \,ds_1\cdots ds_k,
\end{align*}
where the induction hypothesis is applied to \(D_\theta f\).  This proves the
integral identity for all \(k\).

Subtracting \(D_\theta^kf(x)\) from both sides yields
\begin{align*}
 &\left|\frac{\Delta_{t\theta}^kf(x)}{t^k}
 -D_\theta^kf(x)\right| \le
 \sup_{|z-x|\le kt}
 |D^kf(z)-D^kf(x)|_{\mathrm{op}},
\end{align*}
where \(|\cdot|_{\mathrm{op}}\) is the multilinear operator norm.  The
right-hand side tends to zero by continuity of \(D^kf\), independently of
\(\theta\).  This proves both the pointwise and the uniform convergence.
\end{proof}

\begin{theorem}
Let \(k\in\mathbb N\), \(0<q<\infty\), and
\(f\in C_c^k(\mathbb R^n)\).  Then, for every \(x\in\mathbb R^n\),
\[
 \lim_{s\to1^-}\bigl(q(1-s)\bigr)^{1/q}
 \mathcal D_{s,q}^{(k)}f(x)
 =\mathcal A_{n,k,q}(D^kf(x)).
\]
\end{theorem}

\begin{proof}
We write
\[
 T(t):=t^{-kq}\int_{\mathbb S^{n-1}}
 |\Delta_{t\theta}^kf(x)|^q\,d\theta,
 \qquad t>0.
\]
By Lemma~\ref{H-BBM-lim}, the integrand converges uniformly in \(\theta\) to
\(|D^kf(x)[\theta^k]|^q\).  Hence
\[
 \lim_{t\to0^+}T(t)
 =\int_{\mathbb S^{n-1}}|D^kf(x)[\theta^k]|^q\,d\theta
 =\mathcal A_{n,k,q}(D^kf(x))^q.
\]
The same integral representation gives, for \(0<t\le1\),
\[
 \frac{|\Delta_{t\theta}^kf(x)|}{t^k}
 \le\sup_{|z-x|\le k}|D^kf(z)|_{\mathrm{op}},
\]
so \(T\) is locally bounded near the origin. On the other hand,
difference satisfies
\[
 |\Delta_{t\theta}^kf(x)|
 \le\sum_{j=0}^k\binom{k}{j}|f(x+jt\theta)|
 \le2^k\|f\|_{L^{\infty}},
\]
and then
\[
 T(t)\le |\mathbb S^{n-1}|(2^k\|f\|_{L^{\infty}})^q t^{-kq},
 \qquad t\ge1,
\]
which is the condition required in Theorem~\ref{thm:BBM}(I).

Using polar coordinates \(h=t\theta\), we compute
\begin{align*}
 \mathcal D_{s,q}^{(k)}f(x)^q
 &=\int_{\mathbb S^{n-1}}\int_0^\infty
 |\Delta_{t\theta}^kf(x)|^q
 t^{-1-(k-1+s)q}\,dt\,d\theta=\int_0^\infty T(t)t^{q(1-s)-1}\,dt.
\end{align*}
Set \(\gamma=q(1-s)>0\).  Then
\[
 q(1-s)\mathcal D_{s,q}^{(k)}f(x)^q
 =\gamma\int_0^\infty T(t)t^{\gamma-1}\,dt.
\]
As \(s\to1^-\), \(\gamma\to0^+\), and Theorem~\ref{thm:BBM}(I) gives
\[
 \lim_{s\to1^-}q(1-s)\mathcal D_{s,q}^{(k)}f(x)^q
 =\lim_{t\to0^+}T(t)
 =\mathcal A_{n,k,q}(D^kf(x))^q.
\]
Taking the positive \(q\)-th root proves the theorem.
\end{proof}

\subsection{Higher-order BSVY formula}

We first state an auxiliary lemma.

\begin{lemma}\label{lem:difference-maximal}
For every \(k\in\mathbb N\),
\[
 \frac{|\Delta_{t\theta}^kf(x)|}{t^k}
 \le kM_\theta(|D_\theta^kf|)(x),
 \qquad t>0,\quad\theta\in\mathbb S^{n-1}.
\]
\end{lemma}
\begin{proof}
Let
\[
\rho_k
=
\underbrace{
\mathbf{1}_{[0,1]}*\cdots*\mathbf{1}_{[0,1]}
}_{k\text{ factors}}.
\]
Thus, \(\rho_k\) is supported in \([0,k]\), and for every bounded measurable
function \(G\) on \([0,k]\),
\[
\int_{[0,1]^k}
G(s_1+\cdots+s_k)\,ds_1\cdots ds_k
=
\int_0^k G(u)\rho_k(u)\,du.
\]
Moreover,
\[
\rho_k\ge 0,
\qquad
\int_0^k \rho_k(u)\,du=1,
\qquad
\|\rho_k\|_{L^\infty(0,k)}\le 1.
\]
Indeed, the last estimate follows recursively from
\[
\rho_k(u)
=
\int_0^1 \rho_{k-1}(u-s)\,ds
\]
and \(\|\rho_1\|_{L^\infty}=1\).

Applying the preceding integral identity to
\[
G(u):=D_\theta^k f(x+tu\theta)
\]
and using Lemma~\ref{H-BBM-lim}, we obtain
\[
\frac{\Delta_{t\theta}^k f(x)}{t^k}
=
\int_0^k
D_\theta^k f(x+tu\theta)\rho_k(u)\,du.
\]
Using \(0\le \rho_k\le 1\), we find
\begin{align*}
\frac{|\Delta_{t\theta}^k f(x)|}{t^k}
&\le
\int_0^k
|D_\theta^k f(x+tu\theta)|\,du
=
\frac{1}{t}
\int_0^{kt}
|D_\theta^k f(x+r\theta)|\,dr \\
&=
k\left(
\frac{1}{kt}
\int_0^{kt}
|D_\theta^k f(x+r\theta)|\,dr
\right) \le
kM_\theta\bigl(|D_\theta^k f|\bigr)(x).
\end{align*}
This proves the desired estimate.
\end{proof}

For \(\lambda>0\), set
\[
 E_{x,k}(f,\lambda,q,\gamma)
 :=\left\{h\in\mathbb R^n:
 |\Delta_h^kf(x)|>\lambda |h|^{k+\gamma/q}\right\}.
\]

\begin{theorem}
Let \(k\in\mathbb N\), \(f\in C_c^k(\mathbb R^n)\),
\(\gamma\ne0\), and \(0<q<\infty\).  Then
\[
\begin{cases}
\displaystyle
\lim_{\lambda\to\infty}\lambda
\left(\int_{E_{x,k}(f,\lambda,q,\gamma)}
|h|^{\gamma-n}\,dh\right)^{1/q}
=\gamma^{-1/q}\mathcal A_{n,k,q}(D^kf(x)),
&\gamma>0,\\[12pt]
\displaystyle
\lim_{\lambda\to0^+}\lambda
\left(\int_{E_{x,k}(f,\lambda,q,\gamma)}
|h|^{\gamma-n}\,dh\right)^{1/q}
=(-\gamma)^{-1/q}\mathcal A_{n,k,q}(D^kf(x)),
&\gamma<0.
\end{cases}
\]
Moreover,
\[
\begin{split}
&\sup_{\lambda>0}\lambda
\left(\int_{E_{x,k}(f,\lambda,q,\gamma)}
|h|^{\gamma-n}\,dh\right)^{1/q}\le C_k|\gamma|^{-1/q}
\left(\int_{\mathbb S^{n-1}}
 M_\theta(|D_\theta^kf|)(x)^q\,d\theta\right)^{1/q}.
\end{split}
\]
\end{theorem}

\begin{proof}
For each \(\theta\in\mathbb S^{n-1}\), define
\[
 T(t)
 :=\frac{|\Delta_{t\theta}^kf(x)|}{t^k},
 \qquad t>0.
\]
Writing \(h=t\theta\), the level-set condition becomes
\(
 \frac{T_{x,\theta}(t)}{t^{\gamma/q}}>\lambda.
\)
Thus
\[
 \int_{E_{x,k}(f,\lambda,q,\gamma)}|h|^{\gamma-n}\,dh
 =\int_{\mathbb S^{n-1}}
 w_\gamma\left(\left\{t>0:
 \frac{T(t)}{t^{\gamma/q}}>\lambda\right\}\right)d\theta.
\]
By Lemma~\ref{H-BBM-lim},
\[
 \lim_{t\to0^+}T(t)
 =|D_\theta^kf(x)|
 =|D^kf(x)[\theta^k]|.
\]
Lemma~\ref{lem:difference-maximal} gives the uniform-in-\(t\) bound
\[
 0\le T(t)
 \le kM_\theta(|D_\theta^kf|)(x).
\]
For \(f\in C_c^k\), the function on the right is bounded uniformly in
\(\theta\), since
\(|D_\theta^kf|\le |D^kf|_{\mathrm{op}}\).  It is therefore integrable to
any finite power over \(\mathbb S^{n-1}\).

Assume first that \(\gamma>0\).  Theorem~\ref{thm:BSVY}(I), applied for each
fixed \(\theta\), gives
\[
 \lim_{\lambda\to\infty}\lambda^q
 w_\gamma\left(\left\{t>0:
 \frac{T(t)}{t^{\gamma/q}}>\lambda\right\}\right)
 =\frac{|D_\theta^kf(x)|^q}{\gamma}.
\]
Its universal estimate gives the integrable majorant
\[
 0\le\lambda^q
 w_\gamma\left(\left\{t>0:
 \frac{T(t)}{t^{\gamma/q}}>\lambda\right\}\right)
 \le\frac{k^q}{\gamma}
 M_\theta(|D_\theta^kf|)(x)^q.
\]
Dominated convergence in \(\theta\) now yields
\begin{align*}
 &\lim_{\lambda\to\infty}\lambda^q
 \int_{E_{x,k}(f,\lambda,q,\gamma)}|h|^{\gamma-n}\,dh
=\frac1\gamma\int_{\mathbb S^{n-1}}
 |D^kf(x)[\theta^k]|^q\,d\theta
=\frac1\gamma\mathcal A_{n,k,q}(D^kf(x))^q.
\end{align*}
Taking the \(q\)-th root gives the first limit.

When \(\gamma<0\), Theorem~\ref{thm:BSVY}(I) is used with
\(\lambda\to0^+\), and the same dominated-convergence argument gives
\[
 \lim_{\lambda\to0^+}\lambda^q
 \int_{E_{x,k}(f,\lambda,q,\gamma)}|h|^{\gamma-n}\,dh
 =\frac1{-\gamma}\mathcal A_{n,k,q}(D^kf(x))^q.
\]
This proves the second limit.  Finally, integrating the universal estimate in
\(\theta\) for arbitrary \(\lambda>0\) gives
\begin{align*}
 &\lambda^q
 \int_{E_{x,k}(f,\lambda,q,\gamma)}|h|^{\gamma-n}\,dh\le\frac{k^q}{|\gamma|}
 \int_{\mathbb S^{n-1}}
 M_\theta(|D_\theta^kf|)(x)^q\,d\theta.
\end{align*}
Taking the supremum in \(\lambda\) and then the \(q\)-th root proves the
stated weak bound with \(C_k=k\).
\end{proof}

\subsection{Higher-order Frank-type oscillation}

Let \(\mathcal P_{k-1}\) be the vector space of polynomials on \(\mathbb R^n\)
of degree at most \(k-1\), and define
\[
 \operatorname{osc}_k(f;B(x,r))
 :=\inf_{P\in\mathcal P_{k-1}}
 \frac1{|B(x,r)|}\int_{B(x,r)}|f(y)-P(y)|\,dy.
\]
For \(A\in\operatorname{Sym}^k(\mathbb R^n)\), let
\[
 P_A(z):=\frac1{k!}A[z^k]
\]
and set
\[
 \mathfrak C_{n,k}(A)
 :=\inf_{P\in\mathcal P_{k-1}}
 \frac1{|B_1|}\int_{B_1}|P_A(z)-P(z)|\,dz.
\]
Thus \(\mathfrak C_{n,k}(A)\) is the distance of the homogeneous degree-
\(k\) polynomial generated by \(A\) from all lower-order polynomials.

\begin{theorem}
Let \(k\in\mathbb N\) and \(f\in C_{c}^k\).  Then
\[
 \lim_{r\to0^+}r^{-k}\operatorname{osc}_k(f;B(x,r))
 =\mathfrak C_{n,k}(D^kf(x)).
\]
Moreover, there are constants
\(0<a_{n,k}\le b_{n,k}<\infty\) such that
\[
 a_{n,k}|A|_E\le\mathfrak C_{n,k}(A)\le b_{n,k}|A|_E
 \qquad\text{for all }A\in\operatorname{Sym}^k(\mathbb R^n).
\]
\end{theorem}

\begin{proof}
By Taylor's theorem,
\[
 f(x+h)=\sum_{|\alpha| \le k} \frac{D^\alpha f(x)}{\alpha!} \, h^\alpha+R_x(h),
\]
where
\[
 \omega_x(r):=
 \sup_{0<|h|\le r}\frac{|R_x(h)|}{|h|^k}
 \longrightarrow0
 \qquad(r\to0^+).
\]
For \(z\in B_1\), set
$
 \eta_r(z):=r^{-k}R_x(rz),
$
then
$
 \|\eta_r\|_{L^\infty(B_1)}\le\omega_x(r).
$
Since
\[\sum_{|\alpha| \le k} \frac{D^\alpha f(x)}{\alpha!} \, h^\alpha-\sum_{|\alpha| \le k-1} \frac{D^\alpha f(x)}{\alpha!} \, h^\alpha=\frac 1{k!}D^kf(x)[h^k],\]
then
\[
 r^{-k}\Bigl(f(x+rz)-\sum_{|\alpha| \le k-1} \frac{D^\alpha f(x)}{\alpha!} \, (rz)^\alpha\Bigr)
 =P_{D^kf(x)}(z)+\eta_r(z).
\]

We next justify the scaling of the minimizing polynomial.  Given
\(P\in\mathcal P_{k-1}\), define
\[
 Q_{P,r}(z)
 :=r^{-k}\Bigl(P(x+rz)-\sum_{|\alpha| \le k-1} \frac{D^\alpha f(x)}{\alpha!} \, (rz)^\alpha\Bigr).
\]
Then \(Q_{P,r}\in\mathcal P_{k-1}\).  Conversely, for any
\(Q\in\mathcal P_{k-1}\), the polynomial
\[
 P_{Q,r}(y)
 :=\sum_{|\alpha| \le k-1} \frac{D^\alpha f(x)}{\alpha!} \, (y-x)^\alpha+r^kQ\left(\frac{y-x}{r}\right)
\]
belongs to \(\mathcal P_{k-1}\), and the two transformations are inverse to
one another.  Therefore the change of variables \(y=x+rz\) gives the exact
identity
\begin{align*}
 r^{-k}\operatorname{osc}_k(f;B(x,r))
 &=\inf_{Q\in\mathcal P_{k-1}}
 \frac1{|B_1|}\int_{B_1}
 |P_{D^kf(x)}(z)+\eta_r(z)-Q(z)|\,dz.
\end{align*}
Let \(X=L^1(B_1,dz/|B_1|)\) and \(V=\mathcal P_{k-1}\), regarded as a
finite-dimensional closed subspace of \(X\).  The distance to a closed
subspace is one-Lipschitz:
\[
 |\operatorname{dist}_X(u,V)-\operatorname{dist}_X(v,V)|
 \le\|u-v\|_X.
\]
Indeed, for every \(P\in V\),
\[
 \|u-P\|_X\le\|u-v\|_X+\|v-P\|_X,
\]
and taking the infimum first in \(P\), then interchanging \(u\) and \(v\),
proves the claim.  Applying this property with
\(u=P_{D^kf(x)}+\eta_r\) and \(v=P_{D^kf(x)}\), we obtain
\begin{align*}
 &\left|r^{-k}\operatorname{osc}_k(f;B(x,r))
 -\mathfrak C_{n,k}(D^kf(x))\right|\\
 &\qquad\le\|\eta_r\|_X
 \le\|\eta_r\|_{L^\infty(B_1)}
 \le\omega_x(r).
\end{align*}
Since \(\omega_x(r)\to0\), the asserted limit follows.

The two-sided comparison with \(|\cdot|_E\) now follows from compactness of
the Euclidean unit sphere in the finite-dimensional tensor space: the
continuous positive function \(A\mapsto\mathfrak C_{n,k}(A)\) has a positive
minimum and a finite maximum on \(\{|A|_E=1\}\).
\end{proof}

\section{Proofs of Theorem~\ref{thm:unified-frac} and  Corollary \ref{cor:frac}}

\begin{proof}[Proof of Theorem~\ref{thm:unified-frac}]
For \(0<\sigma<1\), the identity
\[
 \frac1{|\Gamma(-\sigma)|}=\frac{\sigma}{\Gamma(1-\sigma)}
\]
follows from \(\Gamma(1-\sigma)=-\sigma\Gamma(-\sigma)\).

For the limit \(\sigma\to1^-\), write
\(T(t)=(f(x)-e^{-tL}f(x))/t\) and \(\gamma=1-\sigma\). Thus,
\[
L^\sigma f(x)=\frac{\sigma}{\Gamma(\gamma)}
\int_0^\infty T(t)t^{\gamma-1}\,dt
=\frac{\sigma}{\gamma\Gamma(\gamma)}
\left[\gamma\int_0^\infty T(t)t^{\gamma-1}\,dt\right].
\]

For \(\sigma\to0^+\), write
\(S(t)=f(x)-e^{-tL}f(x)\) and \(\beta=-\sigma\). Then
\[
L^\sigma f(x)=\frac1{\Gamma(1-\sigma)}
\left[(-\beta)\int_0^\infty S(t)t^{\beta-1}\,dt\right].
\]
Theorem~\ref{thm:BBM}, \(\gamma\Gamma(\gamma)\to1\) and \(\Gamma(1-\sigma)\to1\) give
\eqref{eq:unified-limit-1}.
\end{proof}

\begin{proof}[Proof of Corollary \ref{cor:frac}]
Since the arguments are similar, we give the details for the heat semigroup only.
To this end, it suffices to show that for $L=-\Delta$, there exist constants $a,C>0$ such that the following three properties hold:
\begin{align}
 |f(x)-e^{-tL}f(x)|\le Ct\quad(0<t\le1),
 \label{eq:cor1}
\end{align}
\begin{align}
 |e^{-tL}f(x)|\le Ct^{-a}\quad(t\ge1),
\label{eq:cor2}
\end{align}
and
\begin{align}
\lim_{t\to 0^+} \frac{f(x)-e^{-tL}f(x)}t = Lf(x).
\label{eq:cor3}
\end{align}

For $L=-\Delta$, we have $e^{-tL}=e^{t\Delta}$ and
\begin{align}
f(x)-e^{t\Delta}f(x)
&=
-\int_0^t \frac{d}{dr}e^{r\Delta}f(x)\,dr
=
\int_0^t e^{r\Delta}(-\Delta f)(x)\,dr.
\label{eq:heat-difference-cor}
\end{align}
Consequently,
\begin{equation}
\frac{f(x)-e^{t\Delta}f(x)}{t}
=
\int_0^1 e^{st\Delta}(-\Delta f)(x)\,ds.
\label{eq:heat-average-cor}
\end{equation}
Since $-\Delta f\in\mathcal{S}(\mathbb{R}^n)$, we have
$\lim_{t\to0^+}e^{st\Delta}(-\Delta f)(x)= -\Delta f(x)$
for every $s\in[0,1]$. Moreover,
\[
\left|e^{st\Delta}(-\Delta f)(x)\right|
\le \|\Delta f\|_{L^\infty(\mathbb{R}^n)}.
\]
Thus, applying the dominated convergence theorem in \eqref{eq:heat-average-cor} yields \eqref{eq:cor3}.

The heat kernel representation gives
\[
e^{t\Delta}f(x)
=
\frac{1}{(4\pi t)^{n/2}}
\int_{\mathbb{R}^n}
\exp\left(-\frac{|x-y|^2}{4t}\right)f(y)\,dy.
\]
Hence
\begin{equation}
|e^{t\Delta}f(x)|
\le
(4\pi t)^{-n/2}\|f\|_{L^1(\mathbb{R}^n)}
\to 0
\qquad\text{as }t\to\infty.
\label{eq:heat-decay-cor}
\end{equation}

For $0<t\le1$, it follows from \eqref{eq:heat-difference-cor} that
\[
|f(x)-e^{t\Delta}f(x)|
\le
\int_0^t
\|e^{r\Delta}(-\Delta f)\|_{L^\infty}\,dr
\le
t\|\Delta f\|_{L^\infty}.
\]
Therefore,
\[
\left|
\frac{f(x)-e^{t\Delta}f(x)}{t}
\right|
\le
\|\Delta f\|_{L^\infty},
\qquad 0<t\le1.
\]
For $t\ge1$, \eqref{eq:heat-decay-cor} gives
\[
|f(x)-e^{t\Delta}f(x)|
\le
\|f\|_{L^{\infty}}+(4\pi)^{-n/2}\|f\|_{L^1},
\]
and hence
\[
\left|
\frac{f(x)-e^{t\Delta}f(x)}{t}
\right|
\le
\frac{\|f\|_{L^{\infty}}+(4\pi)^{-n/2}\|f\|_{L^1}}{t}.
\]
Combining the estimates above, we obtain \eqref{eq:cor1} and \eqref{eq:cor2}.
\end{proof}

\section{Proofs of Theorem~\ref{thm:AI} and Corollary \ref{cor:AI}}

Let
\[
K_t(x):=t^{-n}K(x/t),\qquad t>0,
\]
where
\[
K\in L^1,
\qquad
\int_{\mathbb R^n}K(x)\,dx=1.
\]
For every \(\delta>0\), a change of variables gives
\[
\int_{\{|x|>\delta\}}|K_t(x)|\,dx
=
\int_{\{|z|>\delta/t\}}|K(z)|\,dz
\longrightarrow0
\qquad(t\to0^+).
\]
Thus \(\{K_t\}_{t>0}\) satisfies the usual concentration condition for an
approximate identity.
 The following result provides both the pointwise convergence and the
maximal estimate required in Theorem~\ref{thm:AI}, see, for
example, \cite[Corollaries 2.1.12 and 2.1.19]{Grafakos14}.

\begin{lemma}\label{lem:AI-radial-majorant}
Suppose that
\[
|K(x)|\leq\Phi(x)=\phi(|x|),
\qquad
\Phi\in L^1,
\]
where \(\phi:[0,\infty)\to[0,\infty)\) is non-increasing. Then
\[
\lim_{t\to0^+}K_t*f(x)=f(x),
\]
and
\[
\sup_{t>0}|K_t*f(x)|
\leq
\|\Phi\|_{L^1}M^cf(x).
\]
\end{lemma}

\begin{proof}[Proof of Theorem~\ref{thm:AI}]
Set
$
T(t):=|K_t*f(x)|.
$
By Lemma \ref{lem:AI-radial-majorant}, we have
\[
\lim_{t\to 0^+}T(t)=|f(x)|,
\qquad
0\leq T(t)\leq M^cf(x)<\infty.
\]
Applying Theorem~\ref{thm:BSVY}(I) to \(T\), we obtain, for \(\gamma>0\),
\[
\lim_{\lambda\to\infty}
\lambda^q
w_\gamma
\left(
\left\{
t>0:
\frac{|K_t*f(x)|}{t^{\gamma/q}}>\lambda
\right\}
\right)
=
\frac{|f(x)|^q}{\gamma},
\]
whereas, for \(\gamma<0\),
\[
\lim_{\lambda\to0^+}
\lambda^q
w_\gamma
\left(
\left\{
t>0:
\frac{|K_t*f(x)|}{t^{\gamma/q}}>\lambda
\right\}
\right)
=
\frac{|f(x)|^q}{-\gamma}.
\]
Taking the \(q\)-th root proves the two limiting identities. The universal
estimate in Theorem~\ref{thm:BSVY}(I) also gives
\[
|\gamma|^{1/q}
\sup_{\lambda>0}
\lambda
w_\gamma
\left(
\left\{
t>0:
\frac{|K_t*f(x)|}{t^{\gamma/q}}>\lambda
\right\}
\right)^{1/q}
\leq M^cf(x),
\]
which proves the asserted weak quasi-norm bound.
\end{proof}

\begin{proof}[Proof of Corollary \ref{cor:AI}]
We first consider the Poisson kernel, defined by
\[
P(x):=c_n(1+|x|^2)^{-(n+1)/2}
\]
and \(P_t(x)=t^{-n}P(x/t)\). By the normalization of \(c_n\), the function \(P\) is non-negative, radial, non-increasing, integrable, and
$$\int_{\mathbb R^n}P(x)\,dx=1.$$
Hence Lemma~\ref{lem:AI-radial-majorant} applies and yields both
\[
\lim_{t\to0^+}P_t*f(x)=f(x)
\]
and
\[
\sup_{t>0}|P_t*f(x)|
\leq M  f(x).
\]
The same argument to the remaining approximation families. Therefore, all of them satisfy the hypotheses of Theorem~\ref{thm:AI}, and the desired conclusions follow.
\end{proof}

\section{Limiting behavior for harmonic analysis operators}

The classical limiting weak-type formula for maximal and singular integral
operators was established by Janakiraman~\cite{Janakiraman2005}. Subsequent
extensions include weighted, vector-valued, and Dini-kernel variants; see
\cite{HuHuang08,DingLai17,HouGuoWu19,GuoHeWu21}. In this section, we provide a new proof based on the limit formula.

\subsection{Proofs of \ref{thm:weak:T} and Theorems~\ref{thm:strong:T}}

\begin{proof}[Proof of Theorem~\ref{thm:weak:T}]
For almost every \(\theta\in\mathbb S^{n-1}\), define
\[
 F_\theta(r):=r^n|\mathbf Tf(r\theta)|.
\]
Then \(0\le F_\theta(r)\le G_f(\theta)\) and
\(\lim_{r\to +\infty}F_\theta(r)= A_f(\theta)\). Polar coordinates give
\begin{align}
 &W_\gamma\left(\left\{x:
 \frac{|\mathbf Tf(x)|}{|x|^{\gamma-n}}>\lambda\right\}\right)\notag
 =\int_{\mathbb S^{n-1}}
 w_\gamma\left(\left\{r>0:
 \frac{F_\theta(r)}{r^\gamma}>\lambda\right\}\right)d\theta.
\end{align}
Assume \(\gamma>0\). Theorem~\ref{thm:BSVY}(II), with \(q=1\), yields
\[
\lim_{\lambda\to 0^+}\lambda w_\gamma\left(\left\{r:
 \frac{F_\theta(r)}{r^\gamma}>\lambda\right\}\right)
=\frac{A_f(\theta)}\gamma.
\]
Its universal estimate gives
\[
 0\le\lambda w_\gamma\left(\left\{r>0:
 \frac{F_\theta(r)}{r^\gamma}>\lambda\right\}\right)
 \le\frac{G_f(\theta)}\gamma.
\]
By \(G_f\in L^1(\mathbb S^{n-1})\) and dominated convergence theorem, one has
\[
 \lim_{\lambda\to0^+}\lambda W_\gamma
 \left(\left\{x\in\mathbb R^n:
 \frac{|\mathbf Tf(x)|}{|x|^{\gamma-n}}>\lambda\right\}\right)
 =\frac1\gamma\int_{\mathbb S^{n-1}}A_f(\theta)\,d\theta.
\]
For \(\gamma<0\), Theorem~\ref{thm:BSVY}(II) is used as
\(\lambda\to+\infty\), producing the factor \((-\gamma)^{-1}\).
\end{proof}

\begin{proof}[Proof of Theorem~\ref{thm:strong:T}]
Let \(1<p<p_0\) and put \(\alpha=n(p-1)>0\). Set
\[
 (p-1)\|\mathbf Tf\|_{L^p}^p=I_p^{\mathrm{loc}}+I_p^{\infty},
\]
where the two terms correspond to \(|x|<1\) and \(|x|\ge1\). By H\"older's
inequality,
\[
 I_p^{\mathrm{loc}}
 \le(p-1)|B_1|^{1-p/p_0}
 \|\mathbf Tf\|_{L^{p_0}(B_1)}^p
 \longrightarrow0 \quad \text{as} \quad p\to 1^+.
\]
On the other hand, set
$
 F_\theta(r):=r^n|\mathbf Tf(r\theta)|.
$
Then
\begin{align*}
 I_p^{\infty}
 &=(p-1)\int_{\mathbb S^{n-1}}\int_1^\infty
  F_\theta(r)^p r^{-n(p-1)-1}\,dr\,d\theta\\
 &=\frac1n\int_{\mathbb S^{n-1}}
 \left[\alpha\int_1^\infty
 F_\theta(r)^p r^{-\alpha-1}\,dr\right]d\theta.
\end{align*}
Since \(0\le F_\theta(r)\le G_f(\theta)\), then
\[
 \sup_{r\ge1}| F_\theta(r)^p-F_\theta(r)|
 \le\sup_{0\le u\le G_f(\theta)}|u^p-u|.
\]
The right-hand side tends to zero pointwise as \(p\to1^+\) and is dominated
by \(1+G_f(\theta)^{p_0}\). Hence
\begin{equation*}
\lim_{p\to 1^+} \sup_{r\ge1}\|F_\theta(r)^p- F_\theta(r)\|_{L^1(\mathbb S^{n-1})}
 =0.
\end{equation*}
The assumed uniform \(L^{p_0}\)-convergence at infinity also implies
\[
\lim_{R\to+\infty}\sup_{r\ge R}\| F_\theta(r)-A_f\|_{L^1}=0.
\]
Fix \(R>1\),
\[
\lim_{\alpha\to 0^+}\alpha\int_1^Rr^{-\alpha-1}\,dr=\lim_{\alpha\to 0^+}1-R^{-\alpha}=0
\]
and
\[
\lim_{\alpha\to 0^+}\alpha\int_R^\infty r^{-\alpha-1}\,dr=1,
\]
we obtain
\[
\lim_{p\to 1^+} I_p^{\infty}
=\frac1n\int_{\mathbb S^{n-1}}A_f(\theta)\,d\theta.
\]
Together with \(I_p^{\mathrm{loc}}\to0\), this proves the theorem.
\end{proof}

\subsection{Limiting behavior for centered Hardy--Littlewood maximal operator}

Recall that the centered Hardy--Littlewood maximal operator is defined by
\[
M f(x)=\sup_{r>0}|B(x,r)|^{-1}\int_{B(x,r)}|f(y)|\,dy.
\]

\begin{lemma}\label{lem:M}
If \(f\in L^1\) and
\(\supp f\subset B(0,R)\), then
\[
\lim_{|x|\to\infty}|x|^nM^{c}f(x)
=
\frac{1}{v_n}\|f\|_{L^1},
\]
where $v_n=|B(0,1)|$.
\end{lemma}

\begin{proof}
Fix \(x\in\mathbb R^n\) with \(|x|>R\). Since
\[
\supp f\subset B(0,R)\subset B(x,|x|+R),
\]
we have
\[
M f(x)
\ge
\frac{1}{v_n(|x|+R)^n}
\int_{B(x,|x|+R)}|f(y)|\,dy
=
\frac{\|f\|_{L^{1}}}{v_n(|x|+R)^n}.
\]

On the other hand, if \(0<r\le |x|-R\), then
\(
B(x,r)\cap\supp f=\varnothing,
\)
and hence
\[
\int_{B(x,r)}|f(y)|\,dy=0.
\]
If \(r>|x|-R\), then
\[
\frac{1}{|B(x,r)|}
\int_{B(x,r)}|f(y)|\,dy
\le
\frac{\|f\|_{L^{1}}}{v_nr^n}
\le
\frac{\|f\|_{L^{1}}}{v_n(|x|-R)^n}.
\]
Taking the supremum over \(r>0\), we obtain
\[
M f(x)
\le
\frac{\|f\|_{L^{1}}}{v_n(|x|-R)^n}.
\]
Consequently,
\[
\frac{\|f\|_{L^{1}}}{v_n}
\left(\frac{|x|}{|x|+R}\right)^n
\le
|x|^nMf(x)
\le
\frac{\|f\|_{L^{1}}}{v_n}
\left(\frac{|x|}{|x|-R}\right)^n.
\]
Since
\[
\lim_{|x|\to\infty}
\left(\frac{|x|}{|x|\pm R}\right)^n=1,
\]
the desired conclusion follows.
\end{proof}

\begin{remark}
The compact support assumption on \(f\) in Lemma \ref{lem:M} is necessary.
Take \(n=1\) and define
\[
f(x)=\sum_{k=1}^{\infty}\frac{1}{2^{k}}\chi_{[2^{k},\,2^{k}+1]}(x).
\]
Clearly \(\|f\|_{L^{1}}=1\) and \(\supp f\) is unbounded.
Consider the sequence \(x_{k}=2^{k}+\frac12\). Choosing \(r=\frac12\) in the centered maximal function gives
\[
M f(x_{k})\ge\frac{1}{2\cdot \frac12}\int_{2^{k}}^{2^{k}+1}f(y)\,dy
=1\cdot\frac{1}{2^{k}}=\frac{1}{2^{k}}.
\]
Hence
\[
|x_{k}|\,M f(x_{k})\ge\Bigl(2^{k}+\frac12\Bigr)\frac{1}{2^{k}}\to 1\quad\text{as }k\to\infty.
\]
If Lemma \ref{lem:M} were true without the compact support condition, we would have
\[
\lim_{|x|\to\infty}|x|M f(x)=\frac{1}{v_{1}}\|f\|_{L^{1}(\mathbb{R})}=\frac12.
\]
However, it is just contradictory.
\end{remark}
%

\begin{theorem}
If \(f\in L^1\cap L^\infty\) has compact support, then
\[
\lim_{p\to1^+}(p-1)\|Mf\|_{L^p}^p=\|f\|_{L^{1}}.
\]
\end{theorem}

\begin{proof}
By Lemma~\ref{lem:M}, we have
\[
 A_f(\theta)=\frac1{v_n}\|f\|_{L^{1}}.
\]
Therefore Theorem~\ref{thm:strong:T} gives
\[
 \lim_{p\to1^+}(p-1)\|M f\|_{L^p}^p
 =\frac1n\int_{\mathbb S^{n-1}}
 \frac{\|f\|_{L^{1}}}{v_n}\,d\theta.
\]
Since \(|\mathbb S^{n-1}|=nv_n\), the right-hand side equals
\(\|f\|_{L^{1}}\).
\end{proof}

\begin{theorem}\label{thm:weak:M}
If \(f\in L^1\cap L^\infty\) has compact support, then
\[
\begin{cases}
\displaystyle
\lim_{\lambda\to0^+}\lambda W_\gamma
\left(\left\{x:\frac{M f(x)}{|x|^{\gamma-n}}>\lambda\right\}\right)
=\dfrac{n}{\gamma}\|f\|_{L^{1}},&\gamma>0,\\[10pt]
\displaystyle
\lim_{\lambda\to\infty}\lambda W_\gamma
\left(\left\{x:\frac{M f(x)}{|x|^{\gamma-n}}>\lambda\right\}\right)
=\dfrac{n}{-\gamma}\|f\|_{L^{1}},&\gamma<0.
\end{cases}
\]
\end{theorem}

\begin{proof}
It follows from $f\in L^1\cap L^\infty$ that $f\in L^{p_{0}}$ and $Mf\in L^{p_{0}}$, for all $1<p_{0}<\infty$. Then, applying
Lemma~\ref{lem:M} to obtain
\[
 A_f(\theta)=\frac1{v_n}\|f\|_{L^{1}}.
\]
For \(\gamma>0\), Theorem~\ref{thm:weak:T} therefore yields
\begin{align*}
 &\lim_{\lambda\to0^+}\lambda W_\gamma
 \left(\left\{x\in\mathbb R^n:
 \frac{M f(x)}{|x|^{\gamma-n}}>\lambda\right\}\right)\\
 &=\frac1\gamma\int_{\mathbb S^{n-1}}A_f(\theta)\,d\theta=\frac{|\mathbb S^{n-1}|}{\gamma v_n}\|f\|_{L^{1}}
 =\frac n\gamma\|f\|_{L^{1}}.
\end{align*}
For \(\gamma<0\), the same computation is taken as
\(\lambda\to\infty\) and gives \(n\|f\|_{L^{1}}/(-\gamma)\).
\end{proof}

\begin{remark}
The condition \(f \in L^1 \cap L^\infty\) in the estimate
\[
\lim_{p\to 1^+}(p-1)\|M f\|_{L^p}^p =\|f\|_{L^1}
\]
cannot be relaxed to \(f \in L^1\). In fact, there exists \(f \in L^1\) such that \(\|M f\|_{L^p} = \infty\) for every \(p>1\), making the left-hand side infinite while the right-hand side remains finite.

A standard counterexample is given by
\[
f(x)=\frac{1}{x(\log(1/x))^2}\,\mathbf{1}_{(0,1/2)}(x),\qquad x\in\mathbb{R}.
\]
Clearly \(f \in L^1(\mathbb{R})\) because
\[
\int_0^{1/2}\frac{dx}{x(\log(1/x))^2} = \int_{\log 2}^{\infty} \frac{1}{u^2} \, du < \infty.
\]
Since \(f\) is nonnegative and decreasing on \((0,1/2)\), its Hardy--Littlewood maximal function satisfies
\[
M  f(x) \ge \frac{1}{x}\int_0^x f(t)\,dt = \frac{1}{x\log(1/x)},\qquad 0<x<\frac12.
\]
For any \(p>1\), we have
\[
\int_0^{1/2} \bigl(M f(x)\bigr)^p\,dx
\ge \int_0^{1/2} \frac{dx}{x^p(\log(1/x))^p}
= \int_{\log 2}^\infty \frac{e^{(p-1)u}}{u^p}\,du = \infty.
\]
Hence \(\|M f\|_{L^p(\mathbb{R})}^p = \infty\) for all \(p>1\), and consequently
\[
\lim_{p\to 1^+}(p-1)\|M f\|_{L^p(\mathbb{R})}^p = \infty,
\]
which violates the desired inequality. This illustrates the necessity of the additional \(L^\infty\) assumption in the endpoint estimate.
\end{remark}

\begin{remark}
It should be noted that Janakiraman \cite{Janakiraman2005} established
Theorem~\ref{thm:weak:M} for the case $\gamma=n$, with requiring
$f\in L^1$ only. While additional assumptions are imposed on $f$ in the present
work, we not only obtain a weighted version in a more general setting, but
also reveal that the essential reason for the equality to hold is that
\begin{align*}
 \lim_{\lambda\to0^+}\lambda W_\gamma
 \left(\left\{x\in\mathbb R^n:
 \frac{M f(x)}{|x|^{\gamma-n}}>\lambda\right\}\right)
 &=\lim_{t\rightarrow +\infty}\frac1\gamma\int_{\mathbb S^{n-1}}|t|^{n}M (f)(t\theta)\,d\theta.
\end{align*}
\end{remark}

\subsection{Limiting behavior for non-centered Hardy--Littlewood maximal operator}

The non-centered Hardy--Littlewood maximal operator is defined by
\[
M^{'} f(x)
=
\sup_{B\ni x}\frac{1}{|B|}\int_B|f(y)|\,dy,
\]
where the supremum is taken over all Euclidean balls containing \(x\).

\begin{lemma}\label{lem:M-non-centered}
If \(f\in L^1\) and
\(\supp f\subset B(0,R)\), then
\[
\lim_{|x|\to\infty}|x|^n M^{'} f(x)
=
\frac{2^n}{v_n}\|f\|_{L^1}.
\]
\end{lemma}

\begin{proof}
Fix \(x\in\mathbb R^n\) with \(|x|>R\). Set
\[
z_x=\frac{|x|-R}{2|x|}x,
\qquad
r_x=\frac{|x|+R}{2}.
\]
Then \(x\in B(z_x,r_x)\) and
\[
B(0,R)\subset B(z_x,r_x).
\]
Therefore,
\[
M^{'} f(x)
\ge
\frac{1}{v_nr_x^n}
\int_{B(z_x,r_x)}|f(y)|\,dy
=
\frac{2^n\|f\|_{L^1}}
{v_n(|x|+R)^n}.
\]

For the reverse inequality, let \(B(z,r)\) be any ball containing \(x\).
If
\[
B(z,r)\cap\supp f=\varnothing,
\]
then the corresponding average is zero. Otherwise, there exists
\(y\in B(z,r)\cap B(0,R)\). Since \(x,y\in B(z,r)\), we have
\[
|x|-R
\le |x-y|
<2r.
\]
Thus
\[
r>\frac{|x|-R}{2},
\]
and consequently
\[
\frac{1}{|B(z,r)|}\int_{B(z,r)}|f(y)|\,dy
\le
\frac{\|f\|_{L^1}}{v_nr^n}
\le
\frac{2^n\|f\|_{L^1}}
{v_n(|x|-R)^n}.
\]
Taking the supremum over all balls containing \(x\), we obtain
\[
M^{'} f(x)
\le
\frac{2^n\|f\|_{L^1}}
{v_n(|x|-R)^n}.
\]
Hence
\[
\frac{2^n}{v_n}\|f\|_{L^1}
\left(\frac{|x|}{|x|+R}\right)^n
\le
|x|^nM^{'}f(x)
\le
\frac{2^n}{v_n}\|f\|_{L^1}
\left(\frac{|x|}{|x|-R}\right)^n.
\]
Letting \(|x|\to\infty\) and applying the squeeze theorem proves the
desired conclusion.
\end{proof}

\begin{theorem}
If \(f\in L^1\cap L^\infty \) has compact
support, then
\[
\lim_{p\to1^+}(p-1)
\|M^{'} f\|_{L^p }^p
=
2^n\|f\|_{L^1}.
\]
\end{theorem}

\begin{proof}
By Lemma~\ref{lem:M-non-centered}, the directional profile is constant:
\[
A_f(\theta)
=
\frac{2^n}{v_n}\|f\|_{L^1},
\qquad
\theta\in\mathbb S^{n-1}.
\]
Applying Theorem~\ref{thm:strong:T}, we obtain
\begin{align*}
\lim_{p\to1^+}(p-1)
\|M^{'} f\|_{L^p }^p
&=
\frac1n
\int_{\mathbb S^{n-1}}A_f(\theta)\,d\theta=
\frac{2^n|\mathbb S^{n-1}|}{nv_n}
\|f\|_{L^1}.
\end{align*}
Since \(|\mathbb S^{n-1}|=nv_n\), the right-hand side equals
\(
2^n\|f\|_{L^1}.
\)
\end{proof}

\begin{theorem}
If \(f\in L^1\cap L^\infty \) has compact
support, then
\[
\begin{cases}
\displaystyle
\lim_{\lambda\to0^+}\lambda W_\gamma
\left(
\left\{
x\in\mathbb R^n:
\frac{M^{'} f(x)}{|x|^{\gamma-n}}>\lambda
\right\}
\right)
=
\frac{n2^n}{\gamma}
\|f\|_{L^1},
&\gamma>0,\\[12pt]
\displaystyle
\lim_{\lambda\to\infty}\lambda W_\gamma
\left(
\left\{
x\in\mathbb R^n:
\frac{M^{'} f(x)}{|x|^{\gamma-n}}>\lambda
\right\}
\right)
=
\frac{n2^n}{-\gamma}
\|f\|_{L^1},
&\gamma<0.
\end{cases}
\]
\end{theorem}

\begin{proof}
Lemma~\ref{lem:M-non-centered} gives
\[
A_f(\theta)
=
\frac{2^n}{v_n}\|f\|_{L^1}.
\]
For \(\gamma>0\), Theorem~\ref{thm:weak:T} yields
\begin{align*}
&\lim_{\lambda\to0^+}\lambda W_\gamma
\left(
\left\{
x\in\mathbb R^n:
\frac{M^{'} f(x)}{|x|^{\gamma-n}}>\lambda
\right\}
\right)\\
&\qquad
=
\frac1\gamma
\int_{\mathbb S^{n-1}}A_f(\theta)\,d\theta
=
\frac{2^n|\mathbb S^{n-1}|}{\gamma v_n}
\|f\|_{L^1}
=
\frac{n2^n}{\gamma}
\|f\|_{L^1}.
\end{align*}
For \(\gamma<0\), the same computation with
\(\lambda\to\infty\) gives
\(
\frac{n2^n}{-\gamma}
\|f\|_{L^1}.
\)
\end{proof}

\subsection{Limiting behavior for homogeneous Calder\'on--Zygmund operators}

Let
\[
T_\Omega f(x)
=
\operatorname{p.v.}
\int_{\mathbb R^n}
\frac{
\Omega\bigl((x-y)/|x-y|\bigr)
}{
|x-y|^n
}
f(y)\,dy,
\]
where \(\Omega\in C^1(\mathbb S^{n-1})\) satisfies
\[
\int_{\mathbb S^{n-1}}\Omega(\theta)\,d\theta=0.
\]

\begin{lemma}\label{lem:CZ-asymptotic}
Let \(f\in L^1\) have compact support and non-negative. Then
\[
\lim_{r\to\infty}
\sup_{\theta\in\mathbb S^{n-1}}
\Big|
r^n|T_\Omega f(r\theta)|
-
|\Omega(\theta)|\|f\|_{L^{1}}
\Big|
=
0.
\]
In particular,
\[
\lim_{r\to\infty}
r^n|T_\Omega f(r\theta)|
=
|\Omega(\theta)|\|f\|_{L^{1}}
\]
uniformly in \(\theta\in\mathbb S^{n-1}\).
\end{lemma}

\begin{proof}
Suppose that \(\supp f\subset B(0,R)\). For \(r>2R\),
\[
r^nT_\Omega f(r\theta)
=
\int_{\mathbb R^n}
\frac{r^n}{|r\theta-y|^n}
\Omega\left(
\frac{r\theta-y}{|r\theta-y|}
\right)
f(y)\,dy.
\]
Uniformly for
\(\theta\in\mathbb S^{n-1}\) and \(y\in B(0,R)\),
\[
\lim_{r\to +\infty}\frac{r^n}{|r\theta-y|^n}=1
\qquad
\text{and}
\qquad
\lim_{r\to +\infty}\frac{r\theta-y}{|r\theta-y|}=\theta
\]
Since \(\Omega\) is uniformly continuous on
\(\mathbb S^{n-1}\), it follows that
\[
\frac{r^n}{|r\theta-y|^n}
\Omega\left(
\frac{r\theta-y}{|r\theta-y|}
\right)
\longrightarrow
\Omega(\theta)
\]
uniformly in \(\theta\in\mathbb S^{n-1}\) and \(y\in\supp f\).

Moreover, for \(r>2R\),
\[
|r\theta-y|\ge r-R\ge\frac r2,
\]
and hence the integrand is bounded by
\(
2^n\|\Omega\|_{L^\infty(\mathbb S^{n-1})}|f(y)|.
\)
The conclusion follows from the dominated convergence theorem.
\end{proof}

\begin{theorem}\label{thm:strong:CZ}
Let \(f\in L^1\cap L^\infty \) be compactly
supported and non-negative. Then
\[
\lim_{p\to1^+}(p-1)
\|T_\Omega f\|_{L^p }^p
=
\frac{\|f\|_{L^1}}{n}
\int_{\mathbb S^{n-1}}|\Omega(\theta)|\,d\theta.
\]
\end{theorem}

\begin{proof}
Since \(f\in L^1\cap L^\infty \) has compact
support, it belongs to \(L^{p_0}(\mathbb R^n)\) for every \(p_0>1\).
The standard Calder\'on--Zygmund estimates therefore imply that $T_\Omega f\in L^{p_0}$.
By Lemma~\ref{lem:CZ-asymptotic} and the non-negativity of \(f\), we have
\[
A_f(\theta)
=
|\Omega(\theta)|\|f\|_{L^1}.
\]
Theorem~\ref{thm:strong:T} consequently gives
\begin{align*}
\lim_{p\to1^+}(p-1)
\|T_\Omega f\|_{L^p }^p
&=
\frac1n
\int_{\mathbb S^{n-1}}A_f(\theta)\,d\theta
=
\frac{\|f\|_{L^1}}{n}
\int_{\mathbb S^{n-1}}|\Omega(\theta)|\,d\theta.
\end{align*}
\end{proof}

\begin{theorem}\label{thm:weak:CZ}
Under the assumptions of Theorem~\ref{thm:strong:CZ},
\[
\begin{cases}
\displaystyle
\lim_{\lambda\to0^+}\lambda W_\gamma
\left(
\left\{
x\in\mathbb R^n:
\frac{|T_\Omega f(x)|}{|x|^{\gamma-n}}>\lambda
\right\}
\right)
=
\frac{\|f\|_{L^1}}{\gamma}
\int_{\mathbb S^{n-1}}|\Omega(\theta)|\,d\theta,
&\gamma>0,\\[12pt]
\displaystyle
\lim_{\lambda\to\infty}\lambda W_\gamma
\left(
\left\{
x\in\mathbb R^n:
\frac{|T_\Omega f(x)|}{|x|^{\gamma-n}}>\lambda
\right\}
\right)
=
\frac{\|f\|_{L^1}}{-\gamma}
\int_{\mathbb S^{n-1}}|\Omega(\theta)|\,d\theta,
&\gamma<0.
\end{cases}
\]
\end{theorem}

\begin{proof}
Lemma~\ref{lem:CZ-asymptotic} gives
\[
A_f(\theta)
=
|\Omega(\theta)|\|f\|_{L^1}.
\]
For \(\gamma>0\), Theorem~\ref{thm:weak:T} yields
\begin{align*}
&\lim_{\lambda\to0^+}\lambda W_\gamma
\left(
\left\{
x\in\mathbb R^n:
\frac{|T_\Omega f(x)|}{|x|^{\gamma-n}}>\lambda
\right\}
\right)\\
&\qquad
=
\frac1\gamma
\int_{\mathbb S^{n-1}}A_f(\theta)\,d\theta
=
\frac{\|f\|_{L^1}}{\gamma}
\int_{\mathbb S^{n-1}}|\Omega(\theta)|\,d\theta.
\end{align*}
For \(\gamma<0\), the corresponding limit as
\(\lambda\to\infty\) is
\[
\frac{\|f\|_{L^1}}{-\gamma}
\int_{\mathbb S^{n-1}}|\Omega(\theta)|\,d\theta.
\]
\end{proof}

\begin{remark}
The non-negativity assumption on \(f\) is used only to express the
constants in terms of \(\|f\|_{L^1}\). For a general
compactly supported
$
f\in L^1\cap L^\infty,
$
Lemma~\ref{lem:CZ-asymptotic} gives
\[
A_f(\theta)
=
|\Omega(\theta)|
\left|
\int_{\mathbb R^n}f(y)\,dy
\right|.
\]
Accordingly, the conclusions of
Theorems~\ref{thm:strong:CZ} and \ref{thm:weak:CZ} remain valid with
\(\|f\|_{L^1}\) replaced by
\[
\left|
\int_{\mathbb R^n}f(y)\,dy
\right|.
\]
\end{remark}

\subsection{Limiting behavior for directional maximal operators}

Every \(x\in\mathbb R^n\) can be written uniquely as
\[
x=z+s\theta,
\qquad
z\in\theta^\perp,
\quad s\in\mathbb R.
\]
We write
\[
F_\theta(z)
:=
\int_{\mathbb R}|f(z+r\theta)|\,dr,
\qquad z\in\theta^\perp.
\]
By Fubini's theorem,
\[
\int_{\theta^\perp}F_\theta(z)\,dz
=
\|f\|_{L^1}.
\]

\begin{lemma}\label{lem:M-directional}
Let \(f\in L^1\) and suppose that
\(\supp f\subset B(0,R)\). Then, for every
\(z\in\theta^\perp\),
\[
\lim_{s\to-\infty}
|s|M_\theta f(z+s\theta)
=
F_\theta(z).
\]
More precisely, for \(s<-R\),
\[
\frac{F_\theta(z)}{|s|+R}
\le
M_\theta f(z+s\theta)
\le
\frac{F_\theta(z)}{|s|-R}.
\]
Moreover,
\[
\lim_{s\to +\infty}M_\theta f(z+s\theta)=0.
\]
\end{lemma}

\begin{proof}
Fix \(z\in\theta^\perp\) and \(s<-R\). Since
\(\supp f\subset B(0,R)\), we have
\[
f(z+u\theta)=0
\qquad\text{for }u\notin[-R,R].
\]
Taking \(t=R-s=|s|+R\), we obtain
\begin{align*}
M_\theta f(z+s\theta)
&\ge
\frac{1}{R-s}
\int_0^{R-s}|f(z+(s+r)\theta)|\,dr\\
&=
\frac{1}{|s|+R}
\int_s^R|f(z+u\theta)|\,du=
\frac{F_\theta(z)}{|s|+R}.
\end{align*}
For the reverse inequality, suppose that
\[
\int_0^t|f(z+(s+r)\theta)|\,dr>0.
\]
Then the interval \([s,s+t]\) meets \([-R,R]\), and hence
\(
s+t>-R.
\)
It follows that
\(
t>|s|-R.
\)
Therefore,
\[
\frac{1}{t}\int_0^t|f(z+(s+r)\theta)|\,dr
\le
\frac{F_\theta(z)}{t}
\le
\frac{F_\theta(z)}{|s|-R}.
\]
Taking the supremum over \(t>0\), we obtain
\[
M_\theta f(z+s\theta)
\le
\frac{F_\theta(z)}{|s|-R}.
\]
Consequently,
\[
F_\theta(z)\frac{|s|}{|s|+R}
\le
|s|M_\theta f(z+s\theta)
\le
F_\theta(z)\frac{|s|}{|s|-R}.
\]
Letting \(s\to-\infty\) proves the desired limit.

Finally, if \(s>R\), then for every \(r>0\),
\(
|z+(s+r)\theta|
\ge s+r>R.
\)
Thus
\(
f(z+(s+r)\theta)=0,
\)
and therefore
\(
M_\theta f(z+s\theta)=0.
\)
\end{proof}

\begin{remark}
Unlike the centered and non-centered Hardy--Littlewood maximal operators,
\(M_\theta f\) does not exhibit an \(n\)-dimensional radial decay of
order \(|x|^{-n}\).
\end{remark}

\begin{theorem}
If \(f\in L^1\cap L^\infty \) has compact
support, then
\[
\lim_{p\to1^+}
(p-1)\|M_\theta f\|_{L^p }^p
=
\|f\|_{L^1}.
\]
\end{theorem}

\begin{proof}
Suppose that \(\supp f\subset B(0,R)\). Using the orthogonal
decomposition
\[
x=z+s\theta,
\qquad
z\in\theta^\perp,
\quad s\in\mathbb R,
\]
we have
\(
dx=dz\,ds.
\)
Moreover,
\[
F_\theta(z)=0
\qquad\text{if }|z|\ge R,
\]
and
\(
0\le F_\theta(z)
\le
2R\|f\|_{L^\infty }.
\)
In particular, \(F_\theta\) is compactly supported on
\(\theta^\perp\).

For \(s<-2R\), Lemma~\ref{lem:M-directional} gives
\[
\frac{F_\theta(z)}{|s|+R}
\le
M_\theta f(z+s\theta)
\le
\frac{F_\theta(z)}{|s|-R}.
\]
Hence
\begin{align*}
&(p-1)
\int_{\theta^\perp}
\int_{-\infty}^{-2R}
\left(
\frac{F_\theta(z)}{|s|+R}
\right)^p
\,ds\,dz\\
&\qquad\le
(p-1)
\int_{\theta^\perp}
\int_{-\infty}^{-2R}
|M_\theta f(z+s\theta)|^p
\,ds\,dz\\
&\qquad\le
(p-1)
\int_{\theta^\perp}
\int_{-\infty}^{-2R}
\left(
\frac{F_\theta(z)}{|s|-R}
\right)^p
\,ds\,dz.
\end{align*}
Making the change of variables \(u=-s\), we obtain
\[
(3R)^{1-p}
\int_{\theta^\perp}F_\theta(z)^p\,dz
\le
(p-1)
\int_{\theta^\perp}
\int_{-\infty}^{-2R}
|M_\theta f(z+s\theta)|^p
\,ds\,dz
\]
and
\[
(p-1)
\int_{\theta^\perp}
\int_{-\infty}^{-2R}
|M_\theta f(z+s\theta)|^p
\,ds\,dz
\le
R^{1-p}
\int_{\theta^\perp}F_\theta(z)^p\,dz.
\]

For \(|z|\ge R\) or \(s>R\), \(M_\theta f(z+s\theta)=0\). Therefore,
\begin{align*}
0
&\le
(p-1)
\int_{\theta^\perp}
\int_{-2R}^{\infty}
|M_\theta f(z+s\theta)|^p\,ds\,dz\\
&\le
(p-1)
\|f\|_{L^\infty }^p
\left|
\left\{
(z,s)\in\theta^\perp\times\mathbb R:
|z|<R,\ -2R<s<R
\right\}
\right|.
\end{align*}
The right-hand side tends to zero as \(p\to1^+\).

Since \(F_\theta\) is bounded and compactly supported, the dominated
convergence theorem gives
\[
\lim_{p\to1^+}
\int_{\theta^\perp}F_\theta(z)^p\,dz
=
\int_{\theta^\perp}F_\theta(z)\,dz.
\]
Consequently,
\[
\lim_{p\to1^+}
(p-1)\|M_\theta f\|_{L^p }^p
=
\int_{\theta^\perp}F_\theta(z)\,dz.
\]
Finally, by Fubini's theorem,
\[
\int_{\theta^\perp}F_\theta(z)\,dz
=
\int_{\theta^\perp}
\int_{\mathbb R}|f(z+r\theta)|\,dr\,dz
=
\|f\|_{L^1}.
\]
\end{proof}

\begin{theorem}
Let \(f\in L^1\cap L^\infty \) have compact
support. If \(\gamma\ne n-1\), then
\[
\begin{cases}
\displaystyle
\lim_{\lambda\to0^+}
\lambda W_\gamma
\left(
\left\{
x\in\mathbb R^n:
\frac{M_\theta f(x)}
{|x|^{\gamma-n}}
>\lambda
\right\}
\right)
=
\frac{1}{\gamma-n+1}
\|f\|_{L^1},
&\gamma>n-1,\\[14pt]
\displaystyle
\lim_{\lambda\to\infty}
\lambda W_\gamma
\left(
\left\{
x\in\mathbb R^n:
\frac{M_\theta f(x)}
{|x|^{\gamma-n}}
>\lambda
\right\}
\right)
=
\frac{1}{n-1-\gamma}
\|f\|_{L^1},
&\gamma<n-1.
\end{cases}
\]
\end{theorem}

\begin{proof}
Set
\[
F_\theta(z):=\int_{\mathbb R}|f(z+r\theta)|\,dr,
\qquad
\beta:=\gamma-n+1.
\]
Since \(\supp f\subset B(0,R)\),
\[
F_\theta(z)=0\quad\text{for }|z|\ge R,
\qquad
F_\theta(z)\le 2R\|f\|_{L^\infty}.
\]
Writing
\(
x=z-u\theta,
 z\in\theta^\perp, u>R,
\)
Lemma~\ref{lem:M-directional} gives
\[
\frac{F_\theta(z)}{u+R}
\le M_\theta f(z-u\theta)
\le \frac{F_\theta(z)}{u-R}.
\]
Moreover,
\[
|z-u\theta|=(u^2+|z|^2)^{1/2}.
\]

Fix \(\varepsilon\in(0,1)\). There exists \(U>2R\) such that, for
\(u\ge U\) and \(|z|\le R\),
\[
(1-\varepsilon)F_\theta(z)u^{-\beta}
\le
\frac{M_\theta f(z-u\theta)}
{|z-u\theta|^{\gamma-n}}
\le
(1+\varepsilon)F_\theta(z)u^{-\beta},
\]
and
\[
(1-\varepsilon)u^{\beta-1}
\le
|z-u\theta|^{\gamma-n}
\le
(1+\varepsilon)u^{\beta-1}.
\]
Define
\[
I_\lambda
:=
\lambda
\int_{|z|<R}\int_U^\infty
\mathbf 1_{\left\{
\frac{M_\theta f(z-u\theta)}
{|z-u\theta|^{\gamma-n}}>\lambda
\right\}}
|z-u\theta|^{\gamma-n}\,du\,dz.
\]

Assume first that \(\beta>0\). Then
\begin{align*}
&(1-\varepsilon)\lambda
\int_{|z|<R}\int_U^\infty
\mathbf 1_{\{(1-\varepsilon)F_\theta(z)u^{-\beta}>\lambda\}}
u^{\beta-1}\,du\,dz\\
&\qquad\le I_\lambda
\le
(1+\varepsilon)\lambda
\int_{|z|<R}\int_U^\infty
\mathbf 1_{\{(1+\varepsilon)F_\theta(z)u^{-\beta}>\lambda\}}
u^{\beta-1}\,du\,dz.
\end{align*}
For \(a>0\),
\[
\lambda\int_U^\infty
\mathbf 1_{\{aF_\theta(z)u^{-\beta}>\lambda\}}
u^{\beta-1}\,du
=
\frac{1}{\beta}
\bigl(aF_\theta(z)-\lambda U^\beta\bigr)_+.
\]
Hence, letting \(\lambda\to0^+\),
\[
\frac{(1-\varepsilon)^2}{\beta}
\int_{\theta^\perp}F_\theta(z)\,dz
\le
\liminf_{\lambda\to0^+}I_\lambda
\le
\limsup_{\lambda\to0^+}I_\lambda
\le
\frac{(1+\varepsilon)^2}{\beta}
\int_{\theta^\perp}F_\theta(z)\,dz.
\]

The remaining part of the level set is contained in
\[
K_U
:=
\{z+s\theta:|z|<R,\,-U\le s\le R\}.
\]
Since \(\gamma>n-1\), \(W_\gamma(K_U)<\infty\), and therefore
\[
\lambda W_\gamma
\left(
\left\{
x\in K_U:
\frac{M_\theta f(x)}{|x|^{\gamma-n}}>\lambda
\right\}
\right)
\le \lambda W_\gamma(K_U)\longrightarrow0.
\]
Letting \(\varepsilon\to0^+\), we obtain
\[
\lim_{\lambda\to0^+}
\lambda W_\gamma
\left(
\left\{
x:
\frac{M_\theta f(x)}{|x|^{\gamma-n}}>\lambda
\right\}
\right)
=
\frac{1}{\beta}
\int_{\theta^\perp}F_\theta(z)\,dz.
\]

Assume next that \(\beta<0\), and put
\[
\delta:=-\beta=n-1-\gamma>0.
\]
Then
\begin{align*}
&(1-\varepsilon)\lambda
\int_{|z|<R}\int_U^\infty
\mathbf 1_{\{(1-\varepsilon)F_\theta(z)u^\delta>\lambda\}}
u^{-\delta-1}\,du\,dz\\
&\qquad\le I_\lambda\\
&\qquad\le
(1+\varepsilon)\lambda
\int_{|z|<R}\int_U^\infty
\mathbf 1_{\{(1+\varepsilon)F_\theta(z)u^\delta>\lambda\}}
u^{-\delta-1}\,du\,dz.
\end{align*}
Since \(F_\theta\in L^\infty(\theta^\perp)\), for all sufficiently
large \(\lambda\) and \(a\in\{1-\varepsilon,1+\varepsilon\}\),
\[
\lambda\int_U^\infty
\mathbf 1_{\{aF_\theta(z)u^\delta>\lambda\}}
u^{-\delta-1}\,du
=
\frac{aF_\theta(z)}{\delta}.
\]
Consequently,
\[
\frac{(1-\varepsilon)^2}{\delta}
\int_{\theta^\perp}F_\theta(z)\,dz
\le I_\lambda\le
\frac{(1+\varepsilon)^2}{\delta}
\int_{\theta^\perp}F_\theta(z)\,dz
\]
for all sufficiently large \(\lambda\).

On \(K_U\),
\[
\frac{M_\theta f(x)}{|x|^{\gamma-n}}
=
M_\theta f(x)|x|^{n-\gamma}
\le
\|f\|_{L^\infty}
\sup_{x\in K_U}|x|^{n-\gamma}<\infty.
\]
Letting first \(\lambda\to\infty\) and then
\(\varepsilon\to0^+\), we obtain
\[
\lim_{\lambda\to\infty}
\lambda W_\gamma
\left(
\left\{
x:
\frac{M_\theta f(x)}{|x|^{\gamma-n}}>\lambda
\right\}
\right)
=
\frac{1}{\delta}
\int_{\theta^\perp}F_\theta(z)\,dz.
\]

Finally, Fubini's theorem yields
\[
\int_{\theta^\perp}F_\theta(z)\,dz
=
\int_{\theta^\perp}\int_{\mathbb R}
|f(z+r\theta)|\,dr\,dz
=
\|f\|_{L^1}.
\]
Since
\[
\beta=\gamma-n+1,
\qquad
\delta=n-1-\gamma,
\]
the desired conclusions follow.
\end{proof}

\section{Proof of Theorems \ref{thm:DM-1}}

Now, we give the proof of Theorems \ref{thm:DM-1}.

\begin{proof}[Proof of Theorem~\ref{thm:DM-1}]
Let \(M =T^*f(x)\). For every \(\lambda>0\),
\[
 \left\{t>0:\frac{|T_tf(x)|}{t^\gamma}>\lambda\right\}
 \subset
 \left\{t>0:\frac{M }{t^\gamma}>\lambda\right\}
 =\left(0,\left(\frac{M }{\lambda}\right)^{1/\gamma}\right).
\]
Consequently,
\begin{align*}
 \lambda w_\gamma\left(\left\{t:
 \frac{|T_tf(x)|}{t^\gamma}>\lambda\right\}\right)
 &\le\lambda\int_0^{(M /\lambda)^{1/\gamma}}
 t^{\gamma-1}\,dt=\frac{M }{\gamma}.
\end{align*}
Taking the supremum over \(\lambda>0\) gives
\begin{equation}
\label{eq:DM-iterated-upper}
 \left\|\frac{T_tf(x)}{t^\gamma}\right\|_{L_t^{1,\infty}(w_\gamma)}
 \le\frac1\gamma T^*f(x).
\end{equation}
Applying the weak \(L^1(X,m)\) quasi-norm to both sides,
\[
 \left\|\left\|\frac{T_tf}{t^\gamma}\right\|_{L_t^{1,\infty}}
 \right\|_{L_x^{1,\infty}}
 \le\frac1\gamma\|T^*f\|_{L_x^{1,\infty}}.
\]
If \(T^*\) is of weak type \((1,1)\), the last term is at most
\(C\|f\|_{L^{1}(X,m)}/\gamma\).

On the other hand, Theorem~\ref{thm:BSVY}, with \(q=1\) and
\(\gamma>0\), gives
\begin{equation*}
 \lim_{\lambda\to\infty}\lambda w_\gamma
 \left(\left\{t>0:
 \frac{|T_tf(x)|}{t^\gamma}>\lambda\right\}\right)
 ==\lim_{t\to 0^+}\frac{T_tf(x)}{\gamma}=\frac{|f(x)|}{\gamma}.
\end{equation*}
Since the weak quasi-norm is the supremum of the expression on the left,
\[
 \frac{|f(x)|}{\gamma}
 \le\left\|\frac{T_tf(x)}{t^\gamma}\right\|_{L_t^{1,\infty}(w_\gamma)}
 \le\frac{T^*f(x)}\gamma,
\]
where the upper bound is \eqref{eq:DM-iterated-upper}. Taking the outer weak
\(L^1(X,m)\) quasi-norm proves the two-sided estimate.
\end{proof}

\vskip 0.2 true cm
{\bf \color{blue}{Conflict of Interest Statement:}}

\vskip 0.2 true cm

{\bf The authors declare that there is no conflict of interest in relation to this article.}

\vskip 0.2 true cm
{\bf \color{blue}{Data availability statement:}}

\vskip 0.2 true cm

{\bf  Data sharing is not applicable to this article as no data sets are generated
during the current study.}

\vskip 0.2 true cm

\end{document}